\documentclass[twoside]{article}
\usepackage[letterpaper, left=3cm, right=3cm, top=4cm, bottom=2.5cm]{geometry}
\usepackage{graphicx}
\usepackage{amsmath,amssymb,amsthm}
\usepackage{authblk}
\usepackage[utf8]{inputenc}
\usepackage{microtype}
\usepackage{mathrsfs} 
\usepackage{enumitem}
\usepackage{calc}
\usepackage{esint}
\usepackage{fancyhdr}
\usepackage{xcolor}
\usepackage{float}
\usepackage[labelfont=bf]{caption}
\usepackage{algorithm}
\usepackage{algpseudocode}
\floatname{algorithm}{Algorithm}
\usepackage{doi}
\usepackage[numbers]{natbib}
\usepackage{float}
\usepackage{stmaryrd}
\usepackage{mathtools}
\usepackage{booktabs}
\usepackage{colortbl}
\usepackage{tabulary}
\usepackage{diagbox} 
\usepackage{cleveref}
\usepackage{commath}
\usepackage{multirow}
\usepackage{relsize}
\usepackage{lineno}

\usepackage{physics}
\usepackage{amsmath}
\usepackage{tikz}
\usepackage{pgfplots}
\pgfplotsset{compat=1.18}
\usepackage{mathdots}
\usepackage{yhmath}
\usepackage{cancel}
\usepackage{color}
\usepackage{siunitx}
\usepackage{array}
\usepackage{multirow}
\usepackage{amssymb}
\usepackage{gensymb}
\usepackage{tabularx}
\usepackage{extarrows}
\usepackage{booktabs}
\usetikzlibrary{fadings}
\usetikzlibrary{patterns}
\usetikzlibrary{shadows.blur}
\usetikzlibrary{shapes}
\usetikzlibrary{arrows.meta} 

\usepackage{orcidlink}

\newtheorem{theorem}{Theorem}[section]
\numberwithin{equation}{section}
\newtheorem{proposition}[theorem]{Proposition}

\newtheorem{corollary}[theorem]{Corollary}
\newtheorem{remark}[theorem]{Remark}

\newtheorem{assumption}[theorem]{Assumption}

\usepackage{titlesec}
\titleformat{\section}{\normalfont\scshape\centering}{\thesection.}{0.5em}{}
\titleformat*{\subsection}{\itshape}
\titleformat*{\subsubsection}{\itshape}

 \providecommand{\keywords}[1]
 {
 	{\small\emph{Keywords:} #1}
 }
 \providecommand{\MSC}[1]
 {
 	{\small\emph{AMS MSC (2020):~~} #1}
 }

\definecolor{denim}{rgb}{0.08, 0.38, 0.74}
\definecolor{byzantium}{rgb}{0.44, 0.16, 0.39} 
\definecolor{shamrockgreen}{rgb}{0.0, 0.62, 0.38}

\definecolor{qhnavy}{HTML}{244A64}
\definecolor{qhblue}{HTML}{3E718C}
\definecolor{qhteal}{HTML}{28766E}
\definecolor{qhgold}{HTML}{A66A21}
\definecolor{qhgray}{HTML}{6C7780}

\definecolor{qhgrayfill}{HTML}{F5F7F8}
\definecolor{qhbluefill}{HTML}{EFF5F8}
\definecolor{qhgoldfill}{HTML}{FBF5EB}
\definecolor{qhtealfill}{HTML}{EDF6F4}

\usepackage{hyperref}
\hypersetup{
	colorlinks=true,
	linkcolor=denim,
	citecolor = shamrockgreen,
	filecolor=magenta, 
	urlcolor=byzantium,
}

\newcommand{\qhdashedbox}[1]{%
	\tikz[baseline=(box.base)]{
		\node[
			draw=qhblue!85!black,
			line width=0.55pt,
			dash pattern=on 2pt off 1.2pt,
			rounded corners=0.8mm,
			minimum height=4mm,
			inner xsep=1.2mm,
			inner ysep=0pt
		] (box) {%
			\raisebox{-0.5ex}[0pt][0pt]{\smash{$#1$}}%
		};
	}%
}

\usepackage[textsize=small]{todonotes}
\usepackage{soul}

\begin{document}
	\setlength{\abovedisplayskip}{5.5pt}
	\setlength{\belowdisplayskip}{5.5pt}
	\setlength{\abovedisplayshortskip}{5.5pt}
	\setlength{\belowdisplayshortskip}{5.5pt}

	\title{\vspace{-10mm}QH-GEM: Quantum-Hydrodynamic Generative Modeling\thanks{This work is partially supported by the Office of Naval Research (ONR) under Award No. N00014-24-1-2147, the National Science Foundation (NSF) under Grant DMS-2408877, the Air Force Office of Scientific Research (AFOSR) under Award No. FA9550-22-1-0248, and SURE-AI Centre grant 357482, Research Council of Norway.
    }\vspace{-0.5mm}
    }
	\author[1]{Harbir Antil\,\orcidlink{0000-0002-6641-1449}%
	\thanks{Email: \url{hantil@gmu.edu}}}
\author[2]{Alex Kaltenbach\,\orcidlink{0000-0001-6478-7963}%
	\thanks{Email: \url{kaltenbach@math.tu-berlin.de}}}
\author[3]{Sarswati Shah\,\orcidlink{0000-0002-9861-4603}%
	\thanks{Email: \url{sshah66@gmu.edu}\vspace{-0.5mm}}}
	\date{\today\vspace{-2.5mm}} 
	\affil[1,2,3]{\small{Department of Mathematical Sciences and the Center for Mathematics and Artificial Intelligence (CMAI), George Mason University, Fairfax, VA 22030, USA}}
	\affil[2]{\small{Institute of Mathematics, Technical University of Berlin, Stra\ss e des 17.\ Juni 136, 10623 Berlin}\vspace{-0.5mm}}
	\maketitle

	\pagestyle{fancy}
	\fancyhf{}
	\fancyheadoffset{0cm}
	\addtolength{\headheight}{-0.25cm}
	\renewcommand{\headrulewidth}{0pt} 
	\renewcommand{\footrulewidth}{0pt}
	\fancyhead[CO]{\textsc{Quantum-Hydrodynamic Generative Modeling}}
	\fancyhead[CE]{\textsc{H. Antil, A. Kaltenbach, and S. Shah}}
	\fancyhead[R]{\thepage}
	\fancyfoot[R]{}
	
	\begin{abstract}
In this paper, we develop a deterministic, physically constrained
generative framework based on the Madelung formulation of the
free-particle Schrödinger equation. A reference Born probability density and a controllable initial phase function serve as
initial data for the free Madelung~system,
which couples the Born probability density and phase function through the Bohm~quantum~potential, while the phase function determines the hydrodynamic velocity field. Provided that the Born probability density remains positive and the
hydrodynamic velocity field generates a unique global characteristic
flow,  samples drawn from the reference density and transported along the characteristic flow are distributed according to the evolving Born probability density~at~every~time.

As a consequence, randomness enters only through the initial sampling; the
subsequent generation is deterministic and involves neither stochastic
dynamics nor an independently parameterized time-dependent velocity
field. We formulate terminal-time distribution matching as a
PDE-constrained phase-identification \hspace{-0.1mm}problem \hspace{-0.1mm}and \hspace{-0.1mm}derive \hspace{-0.1mm}the \hspace{-0.1mm}underlying
\hspace{-0.1mm}Hamiltonian~\hspace{-0.1mm}and~\hspace{-0.1mm}\mbox{Fisher-information}~\hspace{-0.1mm}\mbox{structure}. For isotropic Gaussian wave packets, we obtain
explicit dynamics and a necessary and sufficient condition for exact
reachability of isotropic Gaussian targets by quadratic initial phase
functions, together with the corresponding sampling map. For a smooth prescribed potential~\mbox{initial}~\mbox{velocity}~field, we
further establish that the characteristic flow approximates the associated
first-order transport map with an $\mathcal{O}(T^2)$ error, both
uniformly and in the $1$- and $2$-Wasserstein distances. A numerical Gaussian benchmark validates the fully discrete forward solver,
while full-grid PDE-constrained phase identification is demonstrated for
asymmetric~\mbox{bimodal}~\mbox{targets}.
These results provide a mathematical proof of concept for
phase-controlled~quantum~\mbox{hydrodynamics}~as~a~\mbox{deterministic}~\mbox{generative}~\mbox{framework}.
\end{abstract}

	
	\MSC{Primary: 35Q41, 35Q93; Secondary: 81Q93, 35Q49, 49M41, 68T05}

    \medskip

	\section{Introduction}\label{sec:introduction}\thispagestyle{empty}\enlargethispage{15mm}\vspace{-1mm} 

\hspace{5mm}Generative modeling plays a central role in modern machine learning,
with prominent applications in language and image generation
(\textit{cf}.~\cite{BrownEtAl2020,RombachEtAl2022,tsimpos2026one}). In the continuous
setting considered here, generative modeling seeks to produce samples
from a target distribution starting from a tractable reference~\mbox{distribution}. Deterministic transport approaches construct a map
between these distributions: normalizing flows parameterize explicitly
invertible transformations
(\textit{cf}.~\cite{RezendeMohamed2015,PapamakariosEtAl2021}),
continuous normalizing flows represent such transformations as the
terminal flows of neural ordinary differential equations
(\textit{cf}.~\cite{ChenEtAl2018}), and flow matching methods learn
time-dependent velocity fields associated with prescribed~\mbox{probability}~paths~(\textit{cf}.~\cite{LipmanEtAl2023}). Diffusion and score-based
models instead employ stochastic forward noising and learned
reverse-time dynamics
(\textit{cf}.~\cite{HoJainAbbeel2020,SongEtAl2021}). Despite their
different formulations, these approaches typically learn a flexibly
parameterized transformation, velocity field, or score function.

In this paper, we study a complementary design principle:
rather than learning a freely~parameterized transport map or
time-dependent velocity field, we fix a physically motivated evolution
law and use the phase of its initial state as the control variable for
the induced transport. The governing law is the \emph{free-particle
Schrödinger equation}, which describes the evolution of the quantum
state of a nonrelativistic particle in the absence of an external
potential, and seeks a \emph{wave function}
$\Psi\colon\mathbb{R}^d\times[0,T]\to\mathbb{C}$~such~that\vspace{-4.5mm}
\begin{subequations}
\label{eq:introduction_free_schroedinger_problem}
\begin{alignat}{2}
	\mathrm{i}\hbar\,\partial_t\Psi
	&=-\smash{\tfrac{\hbar^2}{2m}}\Delta\Psi
	&&\quad\text{ in }\mathbb{R}^d\times(0,T)\,,
	\label{eq:introduction_free_schroedinger_equation}
	\\
	\Psi(\cdot,0)
	&=\Psi_0
	&&\quad\text{ in }\mathbb{R}^d\,.
	\label{eq:introduction_free_schroedinger_initial_condition}\\[-6mm]\notag
\end{alignat}
\end{subequations}
Here, $m>0$ denotes the mass of the particle and
$\smash{\Psi_0\in L^2(\mathbb{R}^d;\mathbb{C})}$ is a normalized initial wave
function, \textit{i.e.},
$\|\Psi_0\|_{L^2(\mathbb{R}^d;\mathbb{C})}=1$. 
Since the free Schrödinger evolution defines a unitary group~on~$L^2(\mathbb{R}^d)$~and,~thus, preserves the $L^2(\mathbb{R}^d;\mathbb{C})$-norm
(\textit{cf}.~\cite[Chap.~2]{Tao2006}), 
$\rho\coloneqq|\Psi|^2\colon \mathbb{R}^d\times[0,T]\to \mathbb{R}_{\ge 0}$ satisfies
$\int_{\mathbb{R}^d}\rho(\cdot,t)\,\mathrm{d}x=1$ for all 
$t\in[0,T]$ and, therefore, is the position probability density
prescribed by the Born rule~(\textit{cf}.~\cite{Born1926}).

If the wave function $\Psi\colon\mathbb{R}^d\times[0,T]\to\mathbb{C}$ is sufficiently regular
and nowhere~vanishing~on~$\mathbb{R}^d\times[0,T]$, resulting in a positive Born probability density $\rho \colon \mathbb{R}^d\times[0,T]\to \mathbb{R}_{>0}$,~the~simple~connectivity~of $\mathbb{R}^d\times[0,T]$ permits the choice of a globally defined phase function $\theta\colon \mathbb{R}^d\times[0,T]\to \mathbb{R}$~such~that\vspace{-0.5mm}
\begin{align}
	\Psi=\sqrt{\rho}
	\exp(\tfrac{\mathrm{i}}{\hbar}\theta)
	\quad\text{ in }
	\mathbb{R}^d\times[0,T]\,.
	\label{eq:introduction_madelung_transform}\\[-6mm]\notag
\end{align}
The \hspace{-0.1mm}change \hspace{-0.1mm}of \hspace{-0.1mm}variables \hspace{-0.1mm}$\Psi\hspace{-0.175em}\mapsto\hspace{-0.175em} (\rho,\theta)$ \hspace{-0.1mm}is \hspace{-0.1mm}referred \hspace{-0.1mm}to \hspace{-0.1mm}as \hspace{-0.1mm}the \hspace{-0.1mm}\emph{Madelung \hspace{-0.1mm}transform}
\hspace{-0.1mm}(\textit{cf}.~\hspace{-0.1mm}\cite[\hspace{-0.1mm}p.~\hspace{-0.1mm}865]{vonRenesse2012};~\hspace{-0.1mm}see~\hspace{-0.1mm}also~\hspace{-0.1mm}\mbox{\cite{CarlesDanchinSaut2012,
ReddigerPoirier2023}}).  

Substitution of the Madelung transform \eqref{eq:introduction_madelung_transform} into
the free Schrödinger equation \eqref{eq:introduction_free_schroedinger_problem} and separation of the real
and imaginary parts results in the \emph{free-particle Madelung system}
(\textit{cf}.~\cite{Madelung1927,CarlesDanchinSaut2012} or Subsection~\ref{subsec:free_madelung_system}).\linebreak Although the
free Schrödinger equation \eqref{eq:introduction_free_schroedinger_problem} is linear, its polar representation
reveals~a~\mbox{nonlinear}~coupling between the Born probability density $\rho\colon \hspace{-0.1em}\mathbb{R}^d\times[0,T]\hspace{-0.1em}\to\hspace{-0.1em} \mathbb{R}_{>0}$ and the
phase function ${\theta\colon \hspace{-0.1em}\mathbb{R}^d\hspace{-0.1em}\times\hspace{-0.1em}[0,T]\hspace{-0.1em}\to\hspace{-0.1em} \mathbb{R}}$:\linebreak More precisely, the imaginary part yields the continuity equation for
the Born probability~\mbox{density}~$\rho \colon \mathbb{R}^d\times[0,T]\to \mathbb{R}_{>0}$, with probability flux
$\rho \mathbf{v}\colon \mathbb{R}^d\times[0,T]\to \mathbb{R}^d$, where $\mathbf{v}\coloneqq \frac{1}{m}\nabla\theta\colon \mathbb{R}^d\times[0,T]\to \mathbb{R}^d$ is the  hydrodynamic velocity field,
and, therefore, expresses local conservation
of probability, while the real part yields the quantum Hamilton--Jacobi~equation for the phase function $\theta\colon \mathbb{R}^d\times[0,T]\to \mathbb{R}$, in which the density enters
through the Bohm quantum potential
$Q_{\mathrm{B}}(\rho)\coloneqq
-\frac{\hbar^2}{2m}\smash{\frac{\Delta\sqrt{\rho}}{\sqrt{\rho}}}\colon \mathbb{R}^d\times[0,T]\to \mathbb{R}$~(\textit{cf}.~\cite{Bohm1952}).

For the transport interpretation, we additionally assume that the
hydrodynamic velocity field $\mathbf{v}\colon\mathbb{R}^d\times[0,T]\to \mathbb{R}^d$ is continuous and globally
Lipschitz with respect to the spatial variable, uniformly~in~time. Then, setting 
$\rho_0\coloneqq\rho(\cdot,0)=|\Psi_0|^2\colon \mathbb{R}^d\to \mathbb{R}_{>0}$, classical ODE theory 
provides a unique global characteristic flow
$X\colon[0,T]\times\mathbb{R}^d\to\mathbb{R}^d$, and the continuity
equation yields the transport representation\vspace{-4.5mm}
\begin{subequations}
\label{eq:introduction_characteristic_transport}
\begin{alignat}{2}
	\partial_tX(t;x)
	&=\mathbf{v}(X(t;x),t)
	&&\quad\text{ for all }(t,x)
	\in(0,T)\times\mathbb{R}^d\,,
	\label{eq:introduction_characteristic_equation}
	\\
	X(0;x)
	&=x
	&&\quad\text{ for all }x\in\mathbb{R}^d\,,
	\label{eq:introduction_characteristic_initial_condition}
	\\
	[X(t;\cdot)]_{\#}(\rho_0\,\mathrm{d}x)
	&=\rho(\cdot,t)\,\mathrm{d}x
	&&\quad\text{ in }\mathcal{P}(\mathbb{R}^d)
	\quad\text{ for all }t\in[0,T]\,,
	\label{eq:introduction_pushforward_representation}\\[-6mm]\notag
\end{alignat}
\end{subequations}
where $\mathcal{P}(\mathbb{R}^d)$ denotes the space of Borel probability measures on $\mathbb{R}^d$.

The free Schrödinger equation \eqref{eq:introduction_free_schroedinger_equation}, thus, determines the
complete time-dependent transport~map and the terminal-time density
$\rho(\cdot,T)\colon\hspace{-0.1em} \mathbb{R}^d\hspace{-0.1em}\to\hspace{-0.1em} \mathbb{R}_{>0}$. Then, the generative problem~is~to~choose~${\theta_0\colon \hspace{-0.1em}\mathbb{R}^d\hspace{-0.1em}\to\hspace{-0.1em} \mathbb{R}}$ such that the terminal-time measure
$\rho(\cdot,T)\,\mathrm{d}x\in \mathcal{P}(\mathbb{R}^d)$ matches, or at least approximates, a prescribed target \hspace{-0.1mm}distribution. \hspace{-0.1mm}Equivalently, \hspace{-0.1mm}if
\hspace{-0.1mm}$X_0\hspace{-0.175em}\sim\hspace{-0.175em}\rho_0\,\mathrm{d}x$, \hspace{-0.1mm}then
\hspace{-0.1mm}\eqref{eq:introduction_pushforward_representation} \hspace{-0.1mm}implies \hspace{-0.1mm}that \hspace{-0.1mm}$X(t;X_0)
	\hspace{-0.175em}\sim\hspace{-0.175em}\rho(\cdot,t)\,\mathrm{d}x$~\hspace{-0.1mm}for~\hspace{-0.1mm}all~\hspace{-0.1mm}${t\hspace{-0.175em}\in\hspace{-0.175em}[0,T]}$.\linebreak
Once \hspace{-0.1mm}the \hspace{-0.1mm}initial \hspace{-0.1mm}phase \hspace{-0.1mm}function \hspace{-0.1mm}$\theta_0\colon \hspace{-0.175em}\mathbb{R}^d\hspace{-0.175em}\to\hspace{-0.175em} \mathbb{R}$ \hspace{-0.1mm}has \hspace{-0.1mm}been \hspace{-0.1mm}fixed, \hspace{-0.1mm}the \hspace{-0.1mm}terminal-time \hspace{-0.1mm}characteristic~\hspace{-0.1mm}map~\hspace{-0.1mm}$X(T;\cdot)\colon \hspace{-0.175em}\mathbb{R}^d$ $\to\hspace{-0.1em} \mathbb{R}^d$ is
deterministic. Randomness enters only through the sampling of the random variable~$X_0\colon\hspace{-0.1em} \Xi\hspace{-0.1em}\to\hspace{-0.1em} \mathbb{R}^d$ on a probability space $(\Xi,\mathcal{F},\mathbb{P})$; \textit{i.e.}, no stochastic process is simulated~during~\mbox{generation}.

The resulting transport \eqref{eq:introduction_characteristic_transport} is physically constrained: its velocity field $\mathbf{v}=\frac{1}{m}\nabla\theta\colon\mathbb{R}^d\times[0,T]\to \mathbb{R}^d$ is a
potential field rather than an independently parameterized
time-dependent vector field, and~its~evolution is coupled to the
Born probability density $\rho \colon \mathbb{R}^d\times[0,T]\to \mathbb{R}_{>0}$ through~the~Bohm~quantum~potential $Q_{\mathrm{B}}(\rho)\colon \mathbb{R}^d\times[0,T]\to \mathbb{R}$ arising from
the Fisher-information contribution to the Hamiltonian
(\textit{cf}.~\cite{vonRenesse2012,KhesinMisiolekModin2019}).\linebreak The induced dynamics inherit the Hamiltonian and time-reversible
structure of the free Schrödinger equation
\eqref{eq:introduction_free_schroedinger_equation}; in particular, the
Bohm quantum potential $Q_{\mathrm{B}}(\rho)\colon \mathbb{R}^d\times[0,T]\to \mathbb{R}$ acts dispersively rather than as a dissipative,
parabolic regularization
(\textit{cf}.~\cite[Sec.~3.2]{CarlesDanchinSaut2012}; see also \cite[Chap.~2]{Tao2006}). The resulting phase-controlled generative mechanism is summarized in
Figure~\ref{fig:introduction_phase_controlled_transport}.\enlargethispage{3.5mm}

\begin{figure}[H]
\centering
\begin{tikzpicture}[
	font=\scriptsize,
	fixed/.style={
		draw=qhnavy!85!black,
		fill=qhnavy!8,
		line width=0.55pt,
		rounded corners=0.8mm,
		text width=2.35cm,
		minimum height=9.5mm,
		align=center,
		inner sep=2.5pt
	},
	control/.style={
		draw=qhgold!90!black,
		fill=qhgold!10,
		line width=0.55pt,
		rounded corners=0.8mm,
		text width=2.35cm,
		minimum height=9.5mm,
		align=center,
		inner sep=2.5pt
	},
	physics/.style={
		draw=qhblue!85!black,
		fill=qhblue!12,
		line width=0.55pt,
		rounded corners=0.8mm,
		text width=2.75cm,
		minimum height=14.5mm,
		align=center,
		inner sep=2.5pt
	},
    physicsdashed/.style={
	draw=qhblue!85!black,
	fill=qhblue!12,
	line width=0.55pt,
	dash pattern=on 2.2pt off 1.4pt,
	line cap=round,
	rounded corners=0.8mm,
	text width=2.75cm,
	minimum height=14.5mm,
	align=center,
	inner sep=2.5pt
},
	output/.style={
		draw=qhteal!85!black,
		fill=qhteal!10,
		line width=0.55pt,
		rounded corners=0.8mm,
		text width=3.15cm,
		minimum height=17mm,
		align=center,
		inner sep=2.5pt
	},
	forward/.style={
		-{Latex[length=1.6mm]},
		draw=qhblue!80!black,
		line width=0.6pt,
		line cap=round
	},
	controlarrow/.style={
		-{Latex[length=1.6mm]},
		draw=qhgold!90!black,
		line width=0.7pt,
		line cap=round
	}
]

\filldraw[
	fill=qhblue!4,
	draw=qhblue!42,
	line width=0.45pt,
	rounded corners=1mm
]
	(1.90,-1.18) rectangle (9.6,1.31);

\node[
	font=\sffamily\scriptsize\bfseries,
	text=qhblue!90!black, 
] at (5.575,1.04)
	{PRESCRIBED QUANTUM DYNAMICS};

\node[fixed] (density) at (-0.1,0.60) {%
	{\color{qhnavy}\sffamily\bfseries
	Reference density}\\[1mm]
	$\rho_0\qhdashedbox{=|\Psi_0|^2}$};

\node[control] (phase) at (-0.1,-0.60) {%
	{\color{qhgold!90!black}\sffamily\bfseries
	Initial phase control}\\[1mm]
	$\theta_0=\theta(\cdot,0)$};

\node[physicsdashed] (dynamics) at (3.68,-0.07) {%
	{\color{qhblue!90!black}\sffamily\bfseries
	Free Schrödinger evolution}\\[1mm]
	$\Psi_0=\sqrt{\rho_0}
	\exp(\frac{\mathrm{i}\theta_0}{\hbar})$\\[1mm]
	$\Psi_0\longmapsto\Psi(\cdot,t)$};

\node[physics,text width=3.55cm] (transport) at (7.46,-0.07) {%
	{\color{qhblue!90!black}\sffamily\bfseries
	Free Madelung evolution and\\characteristic transport}\\[1.5mm]
	$\mathbf{v}=\frac{1}{m}\nabla\theta$,
	$\partial_tX=\mathbf{v}(X,t)$\\[2mm]
    \mbox{$\rho(\cdot,t)\,\mathrm{d}x
	=[X(t;\cdot)]_{\#}
	(\rho_0\,\mathrm{d}x)$}};

\node[output,text width=3.25cm] (terminal) at (12.1,-0.07) {%
	{\color{qhteal!90!black}\sffamily\bfseries
	Deterministic generation}\\[1mm]
	$X_0\sim\rho_0\,\mathrm{d}x$\\[1mm]
	$X_T\coloneqq X(T;X_0)$\\[1mm]
	$[X_T]_{\#}\mathbb{P}=\rho(\cdot,T)\,\mathrm{d}x
	\approx\mu_\star$};

\draw[forward]
	(density.east) to[out=0,in=160] (dynamics.west);

\draw[controlarrow]
	(phase.east) to[out=0,in=200] (dynamics.west);

\draw[forward]
	(dynamics.east) -- (transport.west);

\draw[forward]
	(transport.east) -- (terminal.west);

\path
	(terminal.south) -- ++(0,-0.18)
	coordinate (feedback-right);

\path
	(phase.south |- feedback-right)
	coordinate (feedback-left);

\draw[controlarrow,densely dashed]
	(terminal.south)
	.. controls +(-2.5,-0.90) and +(2.5,-0.90) ..
	node[
		midway,
		above=0.5mm,
		font=\scriptsize,
		text=qhgold!90!black,
		fill=white,
		inner sep=1pt
	]
	{\hspace{-2.5mm}find $\theta_0$ to match $\mu_\star$}
	(phase.south);
\end{tikzpicture}\vspace{-3mm}
\caption{Phase-controlled quantum-hydrodynamic generation obtained
from the free Schrödinger problem
\eqref{eq:introduction_free_schroedinger_problem} via the Madelung
transform \eqref{eq:introduction_madelung_transform}. The reference
Born probability density
$\rho_0\colon\mathbb{R}^d\to\mathbb{R}_{>0}$ is fixed, whereas the
initial phase function
$\theta_0\colon\mathbb{R}^d\to\mathbb{R}$ 
is selected through 
PDE-constrained~phase~control, so that, given an initial sample $X_0\sim\rho_0\,\mathrm{d}x$,  the terminal-time measure
$\rho(\cdot,T)\,\mathrm{d}x\in \mathcal{P}(\mathbb{R}^d)$, equivalently the law of
$X_T\coloneqq X(T;X_0)$,
approximates the target measure $\mu_\star\in \mathcal{P}(\mathbb{R}^d)$. Once $\theta_0\colon\mathbb{R}^d\to\mathbb{R}$ is fixed,
sample transport is deterministic.}  
\label{fig:introduction_phase_controlled_transport}
\end{figure}
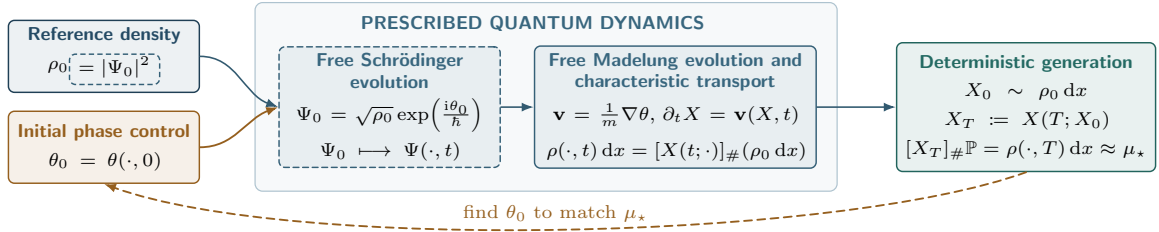\enlargethispage{5mm}\vspace{-6mm}

\paragraph{Relation to previous work and novelty.}
Normalizing flows compose parameterized invertible maps and use the
change-of-variables formula
(\textit{cf}.~\cite{RezendeMohamed2015,PapamakariosEtAl2021});
continuous normalizing flows integrate parameterized time-dependent
velocity fields (\textit{cf}.~\cite{ChenEtAl2018}); and flow matching
learns such fields along prescribed probability paths
(\textit{cf}.~\cite{LipmanEtAl2023}). By contrast, our model
parameterizes neither the terminal-time transport map nor the
time-dependent velocity field. Only the initial phase function is controlled,
while the free Madelung system determines the subsequent transport
through its density-coupled~\mbox{Hamiltonian}~\mbox{dynamics}.

Generative Schrödinger bridges are distinct from the real-time
Schrödinger dynamics considered here, as they describe
entropy-regularized stochastic transport relative to a reference
diffusion
(\textit{cf}.~\cite{Leonard2014,DeBortoliEtAl2021}). 
Other connections between generative models and quantum dynamics
include Q-Flow, which uses normalizing flows to represent and evolve
the Husimi $Q$-function of open quantum systems
(\textit{cf}.~\cite{DuganEtAl2023}); Layden \textit{et al.}, who construct a
Hamiltonian representation of an already trained continuous flow and
a quantum algorithm for its simulation
(\textit{cf}.~\cite{LaydenEtAl2025}); and Lessel's study of the
Wasserstein geometry and optimal transport of a free quantum particle
(\textit{cf}.~\cite{Lessel2025}). Most closely related~is~the~concurrent~work~of~Wang, which learns the score along Bohmian
trajectories and interprets the resulting dynamics as a continuous
normalizing flow for solving the time-dependent Schrödinger equation
(\textit{cf}.~\cite{Wang2026}).

Our direction is converse: rather than learning a prescribed quantum
evolution from trajectories or constructing a quantum Hamiltonian from
a trained flow, we use the physical free-particle Schrödinger evolution
itself as the transport law. The contribution lies not in the Madelung
transform, Bohmian trajectories, or Gaussian wave-packet propagation
individually, but in their combination as a phase-controlled
generative model together with the exact reachability and local
transport results derived~below.\vspace{-2mm}

\paragraph{Main contributions.}
Our principal contributions are:
\begin{itemize}[
noitemsep,
topsep=2pt,
leftmargin=!,
labelwidth=\widthof{$\bullet$}
]
\item \emph{PDE-constrained phase control.}
We formulate deterministic generation as terminal distribution matching
through optimization of the initial phase function of the free Madelung system,
establish its characteristic-flow representation, and identify its
Hamiltonian and Fisher-information structure.
 
\item \emph{Exact Gaussian reachability.}
For quadratic initial phases we derive explicit isotropic Gaussian dynamics, characterize exactly reachable Gaussian targets, and obtain the affine~sampling~map~in~closed~form.

\item \emph{Local nonlinear transport.}
For smooth potential initial velocity fields,
we establish a uniform second-order short-time approximation of the
characteristic flow and corresponding Wasserstein bounds.

\item \emph{Numerical \hspace{-0.1mm}phase \hspace{-0.1mm}identification.}
\hspace{-0.1mm}We \hspace{-0.1mm}validate \hspace{-0.1mm}the \hspace{-0.1mm}solver \hspace{-0.1mm}against \hspace{-0.1mm}exact \hspace{-0.1mm}Gaussian \hspace{-0.1mm}dynamics~\hspace{-0.1mm}and~\hspace{-0.1mm}use~\hspace{-0.1mm}\mbox{full-grid} PDE-constrained optimization to identify initial phases for
asymmetric bimodal targets~in~1D~and~2D.\vspace{-2mm}
\end{itemize}

\paragraph{Organization of the paper.}
In Section~\ref{sec:model}, we introduce the free Madelung problem and
develop its analytical setting, characteristic-flow representation,
Hamiltonian structure, and PDE-constrained phase-identification
formulation. In Section~\ref{sec:exact_gaussian}, we derive explicit
Gaussian wave-packet~\mbox{dynamics}~and~character\-ize exact reachability of isotropic Gaussian targets. In
Section~\ref{sec:local_transport_realization}, we establish local
$\mathcal O(T^2)$-estimates~for the characteristic flow and the
induced probability measures in Wasserstein distances. In Section~\ref{sec:num_validation},~we~va\-lidate the forward
solver on Gaussian dynamics and demonstrate  phase~\mbox{identification}~for~\mbox{bimodal}~\mbox{targets}.\pagebreak

\section{The Free Madelung System as a Generative Transport Model} 
\label{sec:model}

\hspace{5mm}In this section, we formulate the free-particle Madelung system (\textit{cf}.\ \cite{Madelung1927}) and specify the
conditional classical-solution setting used throughout. Then, we 
describe its Hamiltonian structure and deterministic transport
representation, which form the basis of the proposed generative model.

For the remainder of the paper, we adopt units in which $\hbar=1$, while retaining the~particle~mass~${m\hspace{-0.1em}>\hspace{-0.1em}0}$. For a probability density $f\colon \mathbb{R}^d\to\mathbb{R}_{\ge 0}$, $d\in  \mathbb{N}$, \textit{i.e.}, $f\in L^1(\mathbb{R}^d)$ with $\|f\|_{L^1(\mathbb{R}^d)}=1$,~the~\mbox{corresponding} Borel probability
measure $f\,\mathrm{d}x\colon \mathcal{B}(\mathbb{R}^d)\to \mathbb{R}_{\ge 0}$, where $\mathcal{B}(\mathbb{R}^d)$ denotes the Borel $\sigma$-algebra on $\mathbb{R}^d$, is defined by
$(f\,\mathrm{d}x)(A)\coloneqq\int_A{f\,\mathrm{d}x}$ for all 
$A\in\mathcal{B}(\mathbb{R}^d)$. The set of all Borel probability measures
on $\mathbb{R}^d$ is denoted by $\mathcal{P}(\mathbb{R}^d)$; in particular, for each probability density $f\colon  \mathbb{R}^d\to \mathbb{R}_{\ge 0}$,~we~have~that~${f\,\mathrm{d}x\in\mathcal{P}(\mathbb{R}^d)}$.

All subsequent pointwise evaluations of the quantum potential and all
uses of a global characteristic flow are subject to the corresponding
positivity, regularity, and growth assumptions specified below. 

\subsection{Model equations and analytical setting} 
\label{subsec:free_madelung_system}

\hspace{5mm}For a sufficiently regular positive density
$\rho\colon\mathbb{R}^d\times[0,T]\to\mathbb{R}_{>0}$, the associated
\emph{Bohm~quantum~potential}
$Q_{\mathrm{B}}(\rho)\colon
\mathbb{R}^d\times[0,T]\to\mathbb{R}$ is defined by (\textit{cf}.~\cite{Bohm1952})
\begin{align}
	Q_{\mathrm{B}}(\rho)
	\coloneqq
	-\tfrac{1}{2m}
	\smash{\tfrac{\Delta\sqrt{\rho}}{\sqrt{\rho}}}
	\quad\text{in }\mathbb{R}^d\times[0,T]\,.
	\label{eq:bohm_quantum_potential}
\end{align}

The \emph{free Madelung
system}~(\textit{cf}.\ \cite{Madelung1927}), which may be viewed as the
\emph{potential-flow formulation of the quantum hydrodynamic
equations}~(\textit{cf}.\ \cite{CarlesDanchinSaut2012}), seeks a \emph{Born probability density}  $\rho\colon\mathbb{R}^d\times[0,T]\to\mathbb{R}_{>0}$,~\textit{i.e.},  $\rho(\cdot,t)\colon\hspace{-0.15em} \mathbb{R}^d\hspace{-0.15em}\to \hspace{-0.15em}\mathbb{R}_{>0}$ is a probability density for all $t\hspace{-0.15em}\in\hspace{-0.15em} [0,T]$, and a  \emph{phase function}
${\theta\colon\hspace{-0.15em}\mathbb{R}^d\hspace{-0.15em}\times\hspace{-0.15em}[0,T]\hspace{-0.15em}\to\hspace{-0.15em}\mathbb{R}}$~such~that\vspace{-4.5mm}
\begin{subequations}
\label{eq:model}
\begin{alignat}{2}
	\partial_t\rho
	+\tfrac{1}{m}\operatorname{div}(\rho\nabla\theta)
	&=0
	&&\quad\text{in }\mathbb{R}^d\times(0,T)\,,
	\label{eq:model_density}
	\\
	\partial_t\theta
	+\tfrac{1}{2m}|\nabla\theta|^2
	+Q_{\mathrm{B}}(\rho)
	&=0
	&&\quad\text{in }\mathbb{R}^d\times(0,T)\,.
	\label{eq:model_phase}
\end{alignat}
\end{subequations}
Moreover, the free Madelung system \eqref{eq:model} is supplemented by the initial conditions
\begin{subequations}
\label{eq:initial_conditions}
\begin{alignat}{2}
	\rho(\cdot,0)
	&=\rho_0
	&&\quad\text{ in }\mathbb{R}^d\,,
	\label{eq:initial_density}
	\\
	\theta(\cdot,0)
	&=\theta_0
	&&\quad\text{ in }\mathbb{R}^d\,,
	\label{eq:initial_phase}
\end{alignat}
\end{subequations}
where $\rho_0\colon\mathbb{R}^d\to\mathbb{R}_{>0}$ is a probability
density and
$\theta_0\colon\mathbb{R}^d\to\mathbb{R}$ is the prescribed~initial~phase~function. 

The associated \emph{hydrodynamic velocity field} $\mathbf{v}\colon \mathbb{R}^d\times[0,T]\to \mathbb{R}^d$ is given by\vspace{-0.5mm}
\begin{align}
	\mathbf{v}
	\coloneqq\tfrac{1}{m}\nabla\theta
    \quad\text{in }\mathbb{R}^d\times[0,T]\,.
	\label{eq:velocity_field}\\[-6mm]\notag
\end{align}
In particular, the function $\frac{1}{m}\theta \colon \mathbb{R}^d\times[0,T]\to \mathbb{R}$ is a velocity potential, and
equation \eqref{eq:model_density}~is~the~conti\-nuity equation with probability
current $\rho\mathbf{v}\colon \mathbb{R}^d\times[0,T]\to \mathbb{R}^d$, expressing local conservation~of~probability, while
equation~\eqref{eq:model_phase} is the quantum Hamilton--Jacobi
equation in which the phase function is coupled to the Born probability density through the Bohm
quantum potential \eqref{eq:bohm_quantum_potential}.

\subsubsection{Analytical structure and weak formulations}

\hspace{5mm}On every time interval on which the Madelung transform \eqref{eq:introduction_madelung_transform} is
well-defined, the free Madelung system \eqref{eq:model} is equivalent
to the free Schrödinger equation \eqref{eq:introduction_free_schroedinger_equation}, in which case it inherits the~\mbox{Hamiltonian}, time-reversible, \hspace{-0.15mm}and \hspace{-0.15mm}dispersive \hspace{-0.15mm}structure \hspace{-0.15mm}of \hspace{-0.15mm}the \hspace{-0.15mm}Schrödinger \hspace{-0.15mm}evolution
\hspace{-0.15mm}(\textit{cf}.~\hspace{-0.15mm}\cite[\hspace{-0.15mm}Sec.~\hspace{-0.15mm}2]{KhesinMisiolekModin2019},
\hspace{-0.15mm}\cite[\hspace{-0.1mm}Chap.~\hspace{-0.15mm}2]{Tao2006},~\hspace{-0.15mm}and~\hspace{-0.15mm}\cite{CarlesDanchinSaut2012}).

Accordingly, the Bohm quantum potential $Q_{\mathrm{B}}(\rho)\colon
\mathbb{R}^d\times(0,T)\to\mathbb{R}$ acts dispersively and does not
provide a parabolic or dissipative regularization. Moreover, its quotient representation in
\eqref{eq:bohm_quantum_potential} is undefined at vacuum points in $\mathbb{R}^d\times[0,T]$,
\textit{i.e.}, at points $(x,t)\in \mathbb{R}^d\times[0,T]$ at which $\rho(x,t)=0$. Finite-energy weak-solution theories allowing vacuum are available for
quantum~\mbox{hydrodynamic}~\mbox{systems}~(\textit{cf}.~\mbox{\cite{AntonelliMarcati2009,AntonelliMarcati2012}}). More precisely, these formulations retain the continuity equation
\eqref{eq:model_density} in the distributional sense and replace the
phase equation \eqref{eq:model_phase} by the corresponding momentum
balance equation obtained formally from
\eqref{eq:model_density} and the spatial gradient of
\eqref{eq:model_phase}. Their primary variables are the Born probability~density\linebreak
$\rho\colon \hspace{-0.1em}\mathbb{R}^d\times[0,T]\hspace{-0.15em}\to\hspace{-0.15em} \mathbb{R}_{\ge 0}$ \hspace{-0.1mm}and \hspace{-0.1mm}a \hspace{-0.1mm}current \hspace{-0.1mm}$\mathbf{J}\colon \hspace{-0.1em}\mathbb{R}^d\times[0,T]\hspace{-0.15em}\to\hspace{-0.15em} \mathbb{R}^d$, \hspace{-0.1mm}with
\hspace{-0.1mm}$\mathbf{J}\hspace{-0.15em}=\hspace{-0.15em}\rho\mathbf{v}$~in~$\{\rho\hspace{-0.1em}>\hspace{-0.1em}0\}$~\hspace{-0.1mm}(\textit{cf}.~\hspace{-0.1mm}\mbox{\cite[\hspace{-0.1mm}Def.~\hspace{-0.1mm}2.1~\hspace{-0.1mm}\&~\hspace{-0.1mm}Sec.~\hspace{-0.1mm}3]{AntonelliMarcati2012}}). 
The contribution
$-\frac{1}{m}\rho\nabla Q_{\mathrm B}(\rho)$ on the right-hand side of
the corresponding momentum balance~is~rewritten in distributional
divergence form, so that the quotient in \eqref{eq:bohm_quantum_potential} need not be
defined~\mbox{pointwise}~in~${\{\rho\hspace{-0.1em}=\hspace{-0.1em}0\}}$
(\textit{cf}.~\cite[Eq.~(1.3) and Def.~2.1]{AntonelliMarcati2012}). Moreover, the hydrodynamic velocity field can be recovered
from the weak variables \hspace{-0.1mm}through
\hspace{-0.1mm}$\mathbf{v}\hspace{-0.10em}=\hspace{-0.10em}\tfrac{1}{\rho}\mathbf{J}$~\hspace{-0.1mm}in~\hspace{-0.1mm}$\{\rho\hspace{-0.10em}>\hspace{-0.10em}0\}$, \hspace{-0.1mm}and \hspace{-0.1mm}the \hspace{-0.1mm}formulation \hspace{-0.1mm}does
\hspace{-0.1mm}not \hspace{-0.1mm}require \hspace{-0.1mm}a~\hspace{-0.1mm}\mbox{globally}~\hspace{-0.1mm}\mbox{defined}~\hspace{-0.1mm}phase~\hspace{-0.1mm}\mbox{function} (\textit{cf}.~\mbox{\cite[Sec.~3]{AntonelliMarcati2012}},~see~also~\cite{ReddigerPoirier2023}). The global weak-existence result in
\cite{AntonelliMarcati2009} does not include uniqueness, and this weak
solution concept does not furnish the classical characteristic flow ${X\colon\hspace{-0.15em}[0,T]\hspace{-0.15em}\times\hspace{-0.15em}\mathbb{R}^d\hspace{-0.15em}\to\hspace{-0.15em}\mathbb{R}^d}$~\mbox{required}~\mbox{below}.

Since our analysis instead
evaluates the Bohm quantum potential $Q_{\mathrm{B}}(\rho)\colon \mathbb{R}^d\times[0,T]\to \mathbb{R}$ in \eqref{eq:model_phase} pointwise and
uses the unique global classical characteristic flow $X\colon[0,T]\times\mathbb{R}^d\to\mathbb{R}^d$ generated by the hydrodynamic velocity field
$\mathbf{v}\colon \mathbb{R}^d\times[0,T]\to \mathbb{R}^d$ from \eqref{eq:velocity_field}, we impose the~following~standing~hypothesis:

\begin{assumption}[Classical solution and characteristic flow]
\label{ass:classical_solution_setting}
We impose conditions~(\hyperlink{A.1}{A.1}) and
(\hyperlink{A.2}{A.2}) throughout. Conditions
(\hyperlink{A.3}{A.3}) and
(\hyperlink{A.4}{A.4}) are invoked only when the corresponding
derivatives or characteristic flow are used.

\begin{itemize}[noitemsep,leftmargin=!,labelwidth=\widthof{(A.3)},topsep=2pt,font=\upshape
]
	\item[(A.1)] \hypertarget{A.1}{}
    The free Madelung problem \eqref{eq:model}--\eqref{eq:initial_conditions} admits a unique classical solution $(\rho,\theta)\colon\mathbb{R}^d\times [0,T]\to \mathbb{R}_{\ge 0}\times\mathbb{R}$ satisfying
	\begin{align}
		(\rho,\theta)
		&\in
		[
		C^1([0,T];C^0_{\mathrm{loc}}(\mathbb{R}^d))
		\cap
		C^0([0,T];C^2_{\mathrm{loc}}(\mathbb{R}^d))
		]^2\,.
		\label{eq:classical_solution_regularity}
	\end{align}
	Moreover, $\rho(\cdot,t)\colon \mathbb{R}^d\to \mathbb{R}_{\ge 0}$ is a probability density for all
	$t\in[0,T]$;

	\item[(A.2)] \hypertarget{A.2}{}
	The Born probability density remains strictly positive:
	\begin{alignat}{2}
		\rho(x,t)
		&>0
		&&\quad\text{ for all }(x,t)
		\in\mathbb{R}^d\times[0,T]\,;
		\label{eq:strict_positivity}
	\end{alignat}

	\item[(A.3)] \hypertarget{A.3}{}
	Whenever $\nabla Q_{\mathrm{B}}(\rho)$ or
	$\partial_t\mathbf{v}$ is evaluated pointwise, we additionally
	assume that
	\begin{align}
		(\rho,\theta)
		&\in
		C^0([0,T];C^3_{\mathrm{loc}}(\mathbb{R}^d))
		\times
		C^1([0,T];C^1_{\mathrm{loc}}(\mathbb{R}^d))\,;
		\label{eq:additional_solution_regularity}
	\end{align}

	\item[(A.4)] \hypertarget{A.4}{}
	Whenever a global characteristic flow $X\colon[0,T]\times\mathbb{R}^d\to\mathbb{R}^d$ satisfying \eqref{eq:introduction_characteristic_equation} and
\eqref{eq:introduction_characteristic_initial_condition} is used, we assume that
	there exist constants $L_T,C_T>0$ such that
	\begin{subequations}
	\label{eq:velocity_conditions}
	\begin{alignat}{2}
		|\mathbf{v}(x,t)-\mathbf{v}(y,t)|
		&\leq L_T|x-y|
		&&\quad\text{ for all }x,y\in\mathbb{R}^d
		\text{ and }t\in[0,T]\,,
		\label{eq:velocity_global_lipschitz}
		\\
		|\mathbf{v}(x,t)|
		&\leq C_T(1+|x|)
		&&\quad\text{ for all }(x,t)
		\in\mathbb{R}^d\times[0,T]\,.
		\label{eq:velocity_linear_growth}
	\end{alignat}
	\end{subequations}
\end{itemize}
\end{assumption}

Under conditions (\hyperlink{A.1}{A.1}) and
(\hyperlink{A.2}{A.2}) of Assumption~\ref{ass:classical_solution_setting}, the Bohm quantum potential
\eqref{eq:bohm_quantum_potential}~and~all terms in \eqref{eq:model} are
well-defined pointwise. Notice that the strict positivity condition
\eqref{eq:strict_positivity} does not require a uniform lower
bound $\rho\geq c>0$ in $\mathbb{R}^d\times[0,T]$, for some $c\in \mathbb{R}_{>0}$, which, for every $t\in [0,T]$,  would be incompatible with the
integrability requirement $\rho(\cdot,t)\in L^1(\mathbb{R}^d)$  of the probability~\mbox{density}~${\rho(\cdot,t)\colon \mathbb{R}^d\to \mathbb{R}_{\ge 0}}$.

Moreover, Conditions~(\hyperlink{A.1}{A.1}) and
(\hyperlink{A.4}{A.4}) yield the unique global classical
characteristic flow used below (\textit{cf}.\
Proposition~\ref{prop:characteristic_flow_representation}).

\begin{remark}[Relation to Schrödinger well-posedness]
\label{rem:schroedinger_well_posedness}
Assumption~\ref{ass:classical_solution_setting} is a conditional
solution~hypothe\-sis \hspace{-0.1mm}rather \hspace{-0.1mm}than \hspace{-0.1mm}a \hspace{-0.1mm}consequence \hspace{-0.1mm}of \hspace{-0.1mm}a \hspace{-0.1mm}general
\hspace{-0.1mm}well-posedness \hspace{-0.1mm}theorem \hspace{-0.1mm}for \hspace{-0.1mm}the \hspace{-0.1mm}free \hspace{-0.1mm}Madelung~\hspace{-0.1mm}problem~\hspace{-0.1mm}\mbox{\eqref{eq:model}--\eqref{eq:initial_conditions}}.\linebreak
Its relation to standard Schrödinger theory, nevertheless, can  be made
explicit. Let
$\Psi_0\hspace{-0.1em}\coloneqq\hspace{-0.1em}\sqrt{\rho_0}\exp(\mathrm{i}\theta_0) \colon \hspace{-0.1em}\mathbb{R}^d$ $\to \mathbb{C}$ and suppose that
$\Psi_0\in H^s(\mathbb{R}^d;\mathbb{C})$ for some $s>\frac{d}{2}+2$. In the units
$\hbar=1$ used throughout~the~paper,\linebreak the \hspace{-0.1mm}unitary-group
\hspace{-0.1mm}representation \hspace{-0.1mm}of \hspace{-0.1mm}the \hspace{-0.1mm}free \hspace{-0.1mm}Schrödinger \hspace{-0.1mm}evolution
\hspace{-0.1mm}(\textit{cf}.~\hspace{-0.1mm}\cite[Chap.~\hspace{-0.1mm}2]{Cazenave2003};
\hspace{-0.1mm}see~\hspace{-0.1mm}also~\hspace{-0.1mm}\mbox{\cite[Chap.~\hspace{-0.1mm}2]{Tao2006}}) implies that the free
Schrödinger problem
\eqref{eq:introduction_free_schroedinger_problem} admits a unique
global solution satisfying
\begin{align}
	\Psi
	&\in
	C^0(\mathbb{R};H^s(\mathbb{R}^d;\mathbb{C}))
	\cap
	C^1(\mathbb{R};H^{s-2}(\mathbb{R}^d;\mathbb{C}))\,.
	\label{eq:schroedinger_solution_regularity}
\end{align}
Since $\smash{\|\Psi_0\|_{L^2(\mathbb{R}^d)}^2=\|\rho_0\|_{L^1(\mathbb{R}^d)}=1}$ (as $\rho_0\colon \mathbb{R}^d\to \mathbb{R}_{\ge 0}$ is a probability density) and free Schrödinger~evolu\-tion \hspace{-0.1mm}is \hspace{-0.1mm}unitary \hspace{-0.1mm}on
\hspace{-0.1mm}$L^2(\mathbb{R}^d;\mathbb{C})$ \hspace{-0.1mm}and, \hspace{-0.1mm}thus, \hspace{-0.1mm}preserves \hspace{-0.1mm}the \hspace{-0.1mm}$L^2(\mathbb{R}^d;\mathbb{C})$-norm
\hspace{-0.1mm}(\textit{cf}.~\hspace{-0.1mm}\cite[Chap.~\hspace{-0.1mm}2]{Tao2006}),~\hspace{-0.1mm}for~\hspace{-0.1mm}\mbox{every}~\hspace{-0.1mm}${t\hspace{-0.15em}\in\hspace{-0.15em}[0,T]}$, the function
$|\Psi(\cdot,t)|^2\colon\mathbb{R}^d\to\mathbb{R}_{\ge 0}$ is a probability
density. 

If $\Psi$ is nowhere vanishing on
$\mathbb{R}^d\times[0,T]$, the simple connectivity of $\mathbb{R}^d\times[0,T]$ 
permits the choice of a globally defined phase function
$\theta\colon\mathbb{R}^d\times[0,T]\to\mathbb{R}$ (\textit{cf}.~\cite[Sec.~1]{ReddigerPoirier2023};
see also \cite[p.~865]{vonRenesse2012}) such that the \emph{Madelung transform} is specified by
\begin{align}\label{eq:madelung_transform}
	\Psi\mapsto(\rho,\theta)
	\coloneqq(|\Psi|^2,\arg(\Psi))
	\quad\text{ in }
	\mathbb{R}^d\times[0,T]\,.
\end{align}\newpage 
\noindent Here, the branch of $\arg(\Psi)$, which is unique up to an additive
constant $2\pi k$ with $k\in\mathbb{Z}$, is fixed by the initial phase
condition \eqref{eq:initial_phase}.  
Conversely, the \emph{Madelung reconstruction}  
is specified by
\begin{align}\label{eq:madelung_reconstruction}
	(\rho,\theta)\mapsto\Psi
	\coloneqq\sqrt{\rho}
	\exp(\mathrm{i}\theta)
	\quad\text{ in }\mathbb{R}^d\times[0,T]\,.
\end{align}
More precisely, the changes of variables specified by \eqref{eq:madelung_transform} and \eqref{eq:madelung_reconstruction} are the Madelung transform and Madelung reconstruction, respectively, in the
units $\hbar=1$. 

Assume, in addition, that
$\Psi(x,t)\neq0$ for all
$(x,t)\in\mathbb{R}^d\times[0,T]$. On this zero-free~\mbox{space-time}~cylinder, Sobolev embedding and the Madelung transform
\eqref{eq:madelung_transform} yield the regularity in condition
(\hyperlink{A.1}{A.1}),~while~the nonvanishing assumption gives the
strict positivity in condition~(\hyperlink{A.2}{A.2}) of
Assumption~\ref{ass:classical_solution_setting}
(\textit{cf}.~\cite[p.~865]{vonRenesse2012};
see also \cite{CarlesDanchinSaut2012,ReddigerPoirier2023}).
Uniqueness of the Schrödinger solution, together with the~\mbox{continuous}~phase~lift~fixed~by \eqref{eq:initial_phase}, yields uniqueness in this
positive classical solution class. If
$s>\frac{d}{2}+3$, Sobolev embedding applied to $\Psi$ and
$\partial_t\Psi$, together with the fact that $|\Psi|$ is bounded away
from zero on compact subsets, yields the additional regularity in
condition~(\hyperlink{A.3}{A.3}).

The preceding reconstruction is conditional on the nonvanishing
assumption. Standard Schrödinger well-posedness controls Sobolev norms
but does not exclude zeros and therefore does not, by itself, establish
condition~(\hyperlink{A.2}{A.2})
(\textit{cf}.~\cite[Sec.~2]{CarlesDanchinSaut2012}). Even if $\Psi$
is nowhere vanishing, these Sobolev bounds~do~not,~in~general, imply
the global-in-space Lipschitz and linear-growth bounds on $\mathbf v$
required in condition~(\hyperlink{A.4}{A.4}), since derivatives of the
phase involve quotients by $\Psi$, which may approach zero at spatial
infinity.~Consequently,\linebreak conditions~(\hyperlink{A.2}{A.2}) and
(\hyperlink{A.4}{A.4}) remain additional hypotheses not ensured~by~global~Schrödinger~\mbox{well-posedness}. 
\end{remark}

\subsection{Deterministic transport representation}
\label{subsec:deterministic_transport_representation}

\hspace{5mm}The continuity equation \eqref{eq:model_density} describes the Born probability density $\rho\colon \mathbb{R}^d\times [0,T]\to \mathbb{R}_{>0}$
in Eulerian variables, whereas sampling is performed by transporting
particles; the following characteristic representation connects these
two descriptions.

Under condition~(\hyperlink{A.4}{A.4}) of
Assumption~\ref{ass:classical_solution_setting}, let
$X\colon[0,T]\times\mathbb{R}^d\to\mathbb{R}^d$ denote the unique~global~\emph{characteristic flow} generated by the hydrodynamic velocity field $\mathbf{v}\colon \mathbb{R}^d\times  [0,T]\to\mathbb{R}^d$, defined~by~\eqref{eq:velocity_field},~\textit{i.e.},\vspace{-4.5mm}
\begin{subequations}
\label{eq:characteristic_flow}
\begin{alignat}{2}
	\partial_tX(t;x)
	&=\mathbf{v}(X(t;x),t)
	&&\quad\text{ for all }(t,x)
	\in(0,T)\times\mathbb{R}^d\,,
	\label{eq:characteristic_flow_equation}
	\\
	X(0;x)
	&=x
	&&\quad\text{ for all }x\in\mathbb{R}^d\,.
	\label{eq:characteristic_flow_initial_condition}\\[-6mm]\notag
\end{alignat}
\end{subequations}

To describe the action of the characteristic flow $X\colon[0,T]\times\mathbb{R}^d\to\mathbb{R}^d$ on probability
measures, we use the pushforward operation: for a Borel measurable mapping
$F\colon\mathbb{R}^d\to\mathbb{R}^d$ and  
$\mu\in\mathcal{P}(\mathbb{R}^d)$, the \emph{pushforward} of $\mu$ under $F$
is the measure $F_{\#}\mu\in\mathcal{P}(\mathbb{R}^d)$ defined by
\begin{align}
	(F_{\#}\mu)(A)
	\coloneqq
	\mu(F^{-1}(A))
	\quad\text{ for all }A\in\mathcal{B}(\mathbb{R}^d)\,.
	\label{eq:pushforward_definition}
\end{align}
With this notation, the following standard characteristic-flow
representation makes precise the connection between \hspace{-0.1mm}the \hspace{-0.1mm}Eulerian
\hspace{-0.1mm}evolution \hspace{-0.1mm}of \hspace{-0.1mm}the \hspace{-0.1mm}Born \hspace{-0.1mm}probability \hspace{-0.1mm}density \hspace{-0.1mm}and~\hspace{-0.1mm}the~\hspace{-0.1mm}\mbox{particle-based}~\hspace{-0.1mm}\mbox{sampling}~\hspace{-0.1mm}\mbox{procedure}.

\begin{proposition}[Characteristic-flow representation]
\label{prop:characteristic_flow_representation}
Let Assumption~\ref{ass:classical_solution_setting} be satisfied. Then, there exists a unique global characteristic flow $X\colon[0,T]\times\mathbb{R}^d\to\mathbb{R}^d$ generated by the hydrodynamic velocity field $\mathbf{v}\colon \mathbb{R}^d\times  [0,T]\to\mathbb{R}^d$, defined by \eqref{eq:velocity_field}, \textit{i.e.}, a unique global solution of the characteristic-flow problem \eqref{eq:characteristic_flow}, and, for every $t\in [0,T]$, the map 
$X(t;\cdot)\colon\mathbb{R}^d\to\mathbb{R}^d$ is a
$C^1$-diffeomorphism and
\begin{align}
	\rho(\cdot,t)\,\mathrm{d}x
	=[X(t;\cdot)]_{\#}(\rho_0\,\mathrm{d}x)
	\quad\text{ in }\mathcal{P}(\mathbb{R}^d)\,.
	\label{eq:deterministic_pushforward}
\end{align}
Consequently, if $X_0\colon\Xi\to\mathbb{R}^d$ is an $\mathbb{R}^d$-valued random variable on a probability
space $(\Xi,\mathcal{F},\mathbb{P})$ satisfying
\begin{align}
	[X_0]_{\#}\mathbb{P}
	=\rho_0\,\mathrm{d}x
	\quad\text{ in }\mathcal{P}(\mathbb{R}^d)\,,
	\label{eq:initial_sample_distribution}
\end{align}
then, for every $t\in [0,T]$, the transported random variable
$X(t;X_0) \colon\Xi\to\mathbb{R}^d$ satisfies
\begin{align}
	[X(t;X_0)]_{\#}\mathbb{P}
	=\rho(\cdot,t)\,\mathrm{d}x
	\quad\text{ in }\mathcal{P}(\mathbb{R}^d)\,.
	\label{eq:generated_sample_distribution}
\end{align} 
\end{proposition} 

\begin{proof}
By condition~(\hyperlink{A.1}{A.1}) and the global Lipschitz estimate
\eqref{eq:velocity_global_lipschitz} in condition~(\hyperlink{A.4}{A.4}),  the~\mbox{hydrodynamic} velocity \hspace{-0.1mm}field \hspace{-0.1mm}$\mathbf{v}\colon \hspace{-0.175em}\mathbb{R}^d\times[0,T]\hspace{-0.175em}\to\hspace{-0.175em} \mathbb{R}^d$ \hspace{-0.1mm}is
\hspace{-0.1mm}continuous \hspace{-0.1mm}and \hspace{-0.1mm}in \hspace{-0.1mm}its \hspace{-0.1mm}spatial \hspace{-0.1mm}variable
\hspace{-0.1mm}(uniformly \hspace{-0.1mm}with \hspace{-0.1mm}respect~\hspace{-0.1mm}to~\hspace{-0.1mm}${t\hspace{-0.175em}\in\hspace{-0.175em}[0,T]}$) globally Lipschitz continuous. Consequently, the Picard--Lindelöf theorem in its global form implies~that, for each $x\in\mathbb{R}^d$, the characteristic-flow problem
\eqref{eq:characteristic_flow} admits a unique solution $X(\cdot;x)\colon[0,T]\to\mathbb{R}^d$.~Collectively, these solutions define the unique global characteristic
flow
$X\colon\hspace{-0.175em}[0,T]\times\mathbb{R}^d\hspace{-0.175em}\to\hspace{-0.175em}\mathbb{R}^d$~(\textit{cf}.~\mbox{\cite[Cor.~2.6]{Teschl2012}}).
Differentiable dependence on the initial value and uniqueness of the
corresponding backward problem imply that, for every $t\in[0,T]$, $X(t;\cdot)\colon \mathbb{R}^d\to \mathbb{R}^d$ is a $C^1$-diffeomorphism
(\textit{cf}.~\cite[Secs.~2.4~\&~2.6]{Teschl2012}).

By condition~(\hyperlink{A.1}{A.1}) and
\eqref{eq:velocity_field}, we have that
$\mathbf{v}\in
C^0([0,T];C^1_{\mathrm{loc}}(\mathbb{R}^d;\mathbb{R}^d))\subseteq L^1(0,T;W^{1,\infty}_{\mathrm{loc}}
(\mathbb{R}^d;\mathbb{R}^d))$.
Moreover, by the linear-growth estimate \eqref{eq:velocity_linear_growth} in condition (\hyperlink{A.4}{A.4}), we have that $\smash{\frac{|\mathbf{v}|}{1+|\cdot|}}
\in L^1(0,T;L^\infty(\mathbb{R}^d))$.
By condition~(\hyperlink{A.1}{A.1}) and since, for every $t\in[0,T]$, 
$\rho(\cdot,t)\colon \mathbb{R}^d\to \mathbb{R}_{>0}$  is a probability density, the curve $\mu
	\coloneqq
	(\mu_t\coloneqq \rho(\cdot,t)\,\mathrm{d}x)_{t\in[0,T]}
	\in
	C_{\mathrm n}([0,T];\mathcal{P}(\mathbb{R}^d))$
is well-defined. Here,
$C_{\mathrm n}([0,T];\mathcal{P}(\mathbb{R}^d))$ denotes the space of
narrowly continuous curves $t\mapsto\mu_t$, \textit{i.e.}, 
$(t\mapsto\int_{\mathbb{R}^d}\varphi\,\mathrm{d}\mu_t)\in C^0[0,T]$
for every bounded continuous function
$\varphi\in C_{\mathrm{b}}(\mathbb{R}^d)$
(\textit{cf}.~\cite[Sec.~5.1,~Eq.~(5.1.1)]
{AmbrosioGigliSavare2008}).  Indeed, condition~(\hyperlink{A.1}{A.1}) yields
$\rho(\cdot,\widetilde{t})\to\rho(\cdot,t)$ pointwise   $(\widetilde{t}\to t)$. Since these nonnegative
densities all have unit mass, Scheff\'e's lemma  (\textit{cf}.~\cite[Sec.~5.10, p.~55]{Williams1991}) implies convergence in
$L^1(\mathbb{R}^d)$ and, thus, narrow convergence of the associated
probability measures.\linebreak
Testing the pointwise continuity equation
\eqref{eq:model_density} against
$\varphi\in C_{\mathrm{c}}^\infty(\mathbb{R}^d\times[0,T))$, integrating by parts,~and~using the initial condition
\eqref{eq:initial_density}, we obtain $\int_0^T\int_{\mathbb{R}^d}
	(
	\partial_t\varphi
	+\mathbf{v}\cdot\nabla\varphi
	)
	\,\mathrm{d}\mu_t\,\mathrm{d}t
	+
	\int_{\mathbb{R}^d}
	\varphi(\cdot,0)\,\mathrm{d}\mu_0=0$.~Consequently,\linebreak
$\mu\hspace{-0.15em}\in\hspace{-0.15em} C_{\mathrm n}([0,T];\mathcal{P}(\mathbb{R}^d))$ is a
measure-valued distributional solution of
\eqref{eq:model_density} with initial value~${\mu_0\hspace{-0.15em}=\hspace{-0.15em}\rho_0\,\mathrm{d}x}$~and\linebreak the characteristic-flow representation theorem~yields~\eqref{eq:deterministic_pushforward}
(\textit{cf}.~\cite[Sec.~2, Prop.~2.1~\&~Rems.~1.3~\&~2.4]
{AmbrosioCrippa2014}).

Finally, using \eqref{eq:initial_sample_distribution}, the composition
rule for pushforward measures (\textit{cf}.~\cite[Sec.~5.2, Eq.~(5.2.4)]
{AmbrosioGigliSavare2008}), and
\eqref{eq:deterministic_pushforward}, we obtain $[X(t;X_0)]_{\#}\mathbb{P}=[X(t;\cdot)]_{\#}
    ([X_0]_{\#}\mathbb{P})=[X(t;\cdot)]_{\#}(\rho_0\,\mathrm{d}x)=\rho(\cdot,t)\,\mathrm{d}x$ for all  $t\in [0,T]$.
\end{proof}

Given Assumption \ref{ass:classical_solution_setting}, Proposition~\ref{prop:characteristic_flow_representation} provides the
following phase-controlled generative procedure for the free Madelung
dynamics (\textit{cf}.\ Subsection \ref{subsec:free_madelung_system}):

\begin{enumerate}[
	noitemsep,
	topsep=2pt,
	leftmargin=!,
	labelwidth=\widthof{5.},
	font=\itshape
]
\item \emph{Reference and target measures.}\hypertarget{Step 1}{}
We choose a reference Born probability density
$\rho_0\colon\hspace{-0.1em}\mathbb{R}^d\hspace{-0.1em}\to\hspace{-0.1em}\mathbb{R}_{> 0}$~such~that samples from the
corresponding probability measure $\mu_0\coloneqq\rho_0\,\mathrm{d}x\in \mathcal{P}(\mathbb{R}^d)$ 
can be generated efficiently, and prescribe a target probability
measure $\mu_\star\in\mathcal{P}(\mathbb{R}^d)$.

\item \emph{Phase-controlled forward evolution.}\hypertarget{Step 2}{}
For a candidate initial phase function $\theta_0\colon\mathbb{R}^d\to\mathbb{R}$ for which
Assumption~\ref{ass:classical_solution_setting} holds, 
let $(\rho^{\theta_0},\theta^{\theta_0})\colon \mathbb{R}^d\times [0,T]\to \mathbb{R}_{>0}\times \mathbb{R}$ denote the corresponding unique solution
of the free Madelung problem \eqref{eq:model}--\eqref{eq:initial_conditions} and 
let ${\mathbf{v}^{\theta_0}\colon \mathbb{R}^d\times [0,T]\to \mathbb{R}^d}$,~defined~by~\eqref{eq:velocity_field}, be the corresponding hydrodynamic velocity field. 

\item \emph{Characteristic transport.}\hypertarget{Step 3}{} Due to Assumption \ref{ass:classical_solution_setting}, 
Proposition~\ref{prop:characteristic_flow_representation} yields the existence 
of~the~characteristic flow $X^{\theta_0}\colon[0,T]\times\mathbb{R}^d\to\mathbb{R}^d$ solving the characteristic-flow problem \eqref{eq:characteristic_flow} generated~by~$\mathbf{v}^{\theta_0}\colon \mathbb{R}^d\times[0,T]\to\mathbb{R}^d$. Then, the family $X^{\theta_0}(t;\cdot)\colon \mathbb{R}^d\to\mathbb{R}^d$, $t\in [0,T]$,  is the
\emph{characteristic transport flow}, and
Proposition~\ref{prop:characteristic_flow_representation} gives the
induced curve of transported probability measures $\mu^{\theta_0}\in C_{\mathrm n}([0,T];\mathcal{P}(\mathbb{R}^d))$, defined by $\mu^{\theta_0}_t
	\coloneqq \rho^{\theta_0}(\cdot,t)\,\mathrm{d}x
	=
	[X^{\theta_0}(t;\cdot)]_{\#}\mu_0$ in $\mathcal{P}(\mathbb{R}^d)$ for all $t\in [0,T]$.

\item \emph{Phase identification.}\hypertarget{Step 4}{}
The terminal-time characteristic map $X^{\theta_0}(T;\cdot)\colon\hspace{-0.15em} \mathbb{R}^d\hspace{-0.15em}\to\hspace{-0.15em} \mathbb{R}^d$
is the \emph{\mbox{generative}~transport map} associated with the initial phase function $\theta_0\colon \mathbb{R}^d\to \mathbb{R}$.
The \emph{exact phase-control problem} consists of finding an initial
phase function $\theta_0^\star\colon\mathbb{R}^d\to\mathbb{R}$ satisfying the \emph{terminal-time~matching~\mbox{condition}}
\begin{alignat}{2}
    \smash{\mu^{\theta_0^\star}_T=\rho^{\theta_0^\star}(\cdot,T)\,\mathrm{d}x=
	[X^{\theta_0^\star}(T;\cdot)]_{\#}\mu_0
	=\mu_\star
	\quad\text{in }\mathcal{P}(\mathbb{R}^d)\,.}
	\label{eq:phase_control_problem}
\end{alignat}

\item \emph{Deterministic generation.}
After determining the initial
phase function $\theta_0^\star\colon\mathbb{R}^d\to\mathbb{R}$, we draw an initial random variable
$X_0\colon\Xi\to\mathbb{R}^d$ on a probability space $(\Xi,\mathcal{F},\mathbb{P})$ satisfying
${[X_0]_{\#}\mathbb{P}\hspace{-0.1em}=\hspace{-0.1em}\mu_0\hspace{-0.1em}=\hspace{-0.1em}\rho_0\,\mathrm{d}x}$~in~$\mathcal{P}(\mathbb{R}^d)$ and
define the \emph{generated random variable}
$X_T\coloneqq X^{\theta_0^\star}(T;X_0)\colon \Xi\to \mathbb{R}^d$. 
Then, due to Assumption~\ref{ass:classical_solution_setting}, 
Proposition~\ref{prop:characteristic_flow_representation} and
\eqref{eq:phase_control_problem} imply that $X_T\sim\mu_\star$ on $(\Xi,\mathcal F,\mathbb{P})$ since
\begin{align}\label{eq:phase_control_problem.2}
	\smash{[X_T]_{\#}\mathbb{P}= [X^{\theta_0^\star}(T;X_0)]_{\#}\mathbb{P}=
    \mu^{\theta_0^\star}_T=\mu_\star
	\quad\text{ in }\mathcal{P}(\mathbb{R}^d)\,.}
\end{align} 
\end{enumerate}

Therefore, randomness enters only through the initial sample $X_0\colon \Xi\to \mathbb{R}^d$,
whereas,~for~every~$\xi\in \Xi$, the trajectory 
$(t\hspace{-0.1em}\mapsto\hspace{-0.1em} X^{\theta_0^\star}(t;X_0(\xi)))\colon\hspace{-0.1em} [0,T]\hspace{-0.1em}\to \hspace{-0.1em}\mathbb{R}^d$ is deterministic. Neither the~\mbox{characteristic}~\mbox{transport}~flow 
nor its
terminal generative map is prescribed independently: both are generated by the hydrodynamic velocity field, whose evolution is coupled to the Born probability density  through the free
Madelung problem \eqref{eq:model}--\eqref{eq:initial_conditions}. The Hamiltonian structure underlying this coupling is described~in~the~next~\mbox{subsection}.
 
\begin{figure}[H]
\centering
\begingroup

\setlength{\fboxsep}{1.3pt}
\setlength{\fboxrule}{0.45pt}

\newcommand{\qhcardcontent}[4]{%
	\parbox[c][25mm][t]{2.575cm}{%
		\setlength{\parskip}{0pt}%
		\centering\scriptsize

		\vspace*{0.7mm}%
		\colorbox{#1}{%
			\makebox[4.2mm][c]{%
				\color{white}\sffamily\bfseries #2}}%
		\par
		\vspace{0.45mm}%

		\parbox[c][6.2mm][c]{2.30cm}{%
			\centering
			{\color{#1}\sffamily\bfseries #3}%
		}%
		\par
		\vspace{0.25mm}%

		{\color{#1!55}\rule{1.72cm}{0.35pt}}%
		\par

		\vfill
		\parbox[c]{2.30cm}{%
			\centering #4%
		}%
		\par
		\vfill
	}%
}

\newcommand{\qhoneline}[1]{%
	\makebox[0pt][c]{$#1$}%
}

\begin{tikzpicture}[
	font=\scriptsize,
	card/.style={
		line width=0.55pt,
		rounded corners=0.8mm,
		inner sep=\dimexpr
			\fboxsep+\fboxrule-0.275pt
		\relax,
		outer sep=0pt
	},
	forward/.style={
		-{Latex[length=1.6mm]},
		draw=qhblue!80!black,
		line width=0.6pt,
		line cap=round
	}
]

\node[
	card,
	draw=qhnavy!85!black,
	fill=qhnavy!8
] (step1) at (0,0) {%
	\qhcardcontent
		{qhnavy!85!black}
		{1}
		{Reference and target}
		{%
			$\mu_0\coloneqq\rho_0\,\mathrm{d}x$,\par
			$\mu_\star\in\mathcal{P}(\mathbb{R}^d)$
		}%
};

\node[
	card,
	draw=qhblue!85!black,
	fill=qhblue!12
] (step2) at (32.18mm,0) {%
	\qhcardcontent
		{qhblue!85!black}
		{2}
		{Forward evolution}
		{%
			\qhoneline{%
				\theta_0\mapsto
				(\rho^{\theta_0},
				\theta^{\theta_0},
				\mathbf{v}^{\theta_0})%
			}%
		}%
};

\node[
	card,
	draw=qhblue!85!black,
	fill=qhblue!12
] (step3) at (64.36mm,0) {%
	\qhcardcontent
		{qhblue!85!black}
		{3}
		{Characteristic transport}
		{%
			\qhoneline{%
				\mu_t^{\theta_0}
				=
				[X^{\theta_0}(t;\cdot)]_{\#}\mu_0
			}%
		}%
};

\node[
	card,
	draw=qhgold!90!black,
	fill=qhgold!10
] (step4) at (96.54mm,0) {%
	\qhcardcontent
		{qhgold!90!black}
		{4}
		{Terminal-time matching}
		{%
			$\text{find }\theta_0^\star
			\mathpunct{\mathchar"603A}
			\mu_T^{\theta_0^\star}=\mu_\star$
		}%
};

\node[
	card,
	draw=qhteal!85!black,
	fill=qhteal!10
] (step5) at (128.72mm,0) {%
	\qhcardcontent
		{qhteal!85!black}
		{5}
		{Deterministic generation}
		{%
			$X_0\sim\mu_0$,\par
			\qhoneline{%
				X_T\coloneqq
				X^{\theta_0^\star}(T;X_0),
			}%
			\par
			$X_T\sim\mu_\star$
		}%
};

\draw[forward]
	(step1.east) -- (step2.west);

\draw[forward]
	(step2.east) -- (step3.west);

\draw[forward]
	(step3.east) -- (step4.west);

\draw[forward]
	(step4.east) -- (step5.west);

\node[
	font=\sffamily\scriptsize\bfseries,
	text=qhblue!90!black
] at (48.27mm,20.4mm)
	{PHASE IDENTIFICATION};

\draw[
	draw=qhblue!42,
	line width=0.45pt
]
	(-10.23mm,15.5mm)
	--
	(106.77mm,15.5mm);

\node[
	font=\sffamily\scriptsize\bfseries,
	text=qhteal!90!black
] at (128.72mm,20.4mm)
	{GENERATION};

\draw[
	draw=qhteal!42,
	line width=0.45pt
]
	(115.82mm,15.5mm)
	--
	(141.62mm,15.5mm);

\end{tikzpicture}

\endgroup
\caption{Phase-controlled generative procedure.
Steps~\protect\hyperlink{Step 1}{1}--\protect\hyperlink{Step 4}{4}
seek an initial phase function
$\theta_0^\star\colon \mathbb{R}^d\to\mathbb{R}$
satisfying the terminal-time matching condition
\eqref{eq:phase_control_problem}; then, once such a phase function
has~been~found, Step~\protect\hyperlink{Step 5}{5} maps an initial
sample $X_0\sim\mu_0$ deterministically to a (terminal-time)
generated sample $X_T\sim\mu_\star$ via the characteristic-flow
representation \eqref{eq:deterministic_pushforward} in
Proposition~\ref{prop:characteristic_flow_representation}.}
\label{fig:phase_controlled_generative_procedure}
\end{figure}
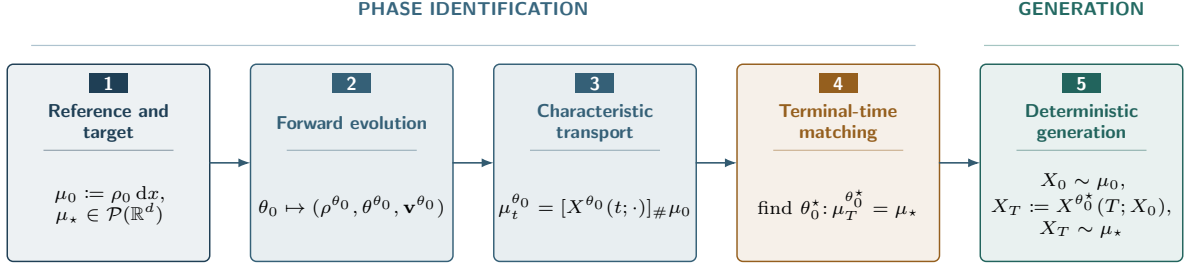

\subsection{Hamiltonian and Fisher-information structure}
\label{subsec:hamiltonian_structure}

\hspace{5mm}The characteristic-flow representation
\eqref{eq:deterministic_pushforward} and the resulting sampling
identity \eqref{eq:generated_sample_distribution} follow from the
continuity equation \eqref{eq:model_density} together with conditions
(\hyperlink{A.1}{A.1}) and (\hyperlink{A.4}{A.4}) of
Assumption~\ref{ass:classical_solution_setting},~which~ensure~the
existence and uniqueness of the global characteristic flow. In fact,
the same identities hold for any sufficiently regular velocity field
in a continuity equation. Therefore, they justify deterministic
sampling but do not encode the phase-density coupling through the
Bohm quantum potential that distinguishes the free Madelung system
\eqref{eq:model}. 

This structure enters through the coupled determination of the
hydrodynamic velocity field $\mathbf{v}\colon \mathbb{R}^d\times [0,T]\to \mathbb{R}^d$ and the phase function $\theta \colon \mathbb{R}^d\times [0,T]\to \mathbb{R}$: definition \eqref{eq:velocity_field}
constrains the
hydrodynamic velocity field to a potential field, while the  quantum Hamilton--Jacobi
equation 
\eqref{eq:model_phase} couples the phase function to the Born probability density
$\rho\colon \mathbb{R}^d\times [0,T]\to \mathbb{R}_{>0}$ through the Bohm~quantum~potential
$Q_{\mathrm B}(\rho)\colon \mathbb{R}^d\times [0,T]\to \mathbb{R}$, defined by \eqref{eq:bohm_quantum_potential}. The Hamiltonian formulation below makes this
coupling precise by identifying the Bohm quantum potential as the variational
derivative, with respect to the density, of the Fisher-information
contribution to the Hamiltonian.

First, for a sufficiently regular Born probability density
$\rho\colon\mathbb{R}^d\to\mathbb{R}_{>0}$ satisfying
$\nabla\sqrt{\rho}\in L^2(\mathbb{R}^d;\mathbb{R}^d)$, its \emph{Fisher
information} (see, \textit{e.g.}, \cite{GianazzaSavareToscani2009,vonRenesse2012}) is defined by 
\begin{align}
	\smash{\mathcal{I}(\rho)
	\coloneqq
    \|\sqrt{\rho}\nabla\log\rho\|_{L^2(\mathbb{R}^d)}^2 
	=\smash{\bigl\|\tfrac{1}{\sqrt{\rho}}\nabla\rho\bigr\|_{L^2(\mathbb{R}^d)}^2} 
	=4\|\nabla\sqrt{\rho}\|_{L^2(\mathbb{R}^d)}^2\,.} 
	\label{eq:fisher_information}
\end{align}

Then, for a Born probability density
$\rho\colon\hspace{-0.15em}\smash{\mathbb{R}^d}\hspace{-0.15em}\to\hspace{-0.15em}\mathbb{R}_{>0}$ as above and for a sufficiently regular~phase~\mbox{function}
$\theta\colon\mathbb{R}^d\to\mathbb{R}$ satisfying
$\sqrt{\rho}\nabla\theta\in L^2(\mathbb{R}^d;\mathbb{R}^d)$, 
we define the \emph{Hamiltonian} (see, \textit{e.g.},  \cite{vonRenesse2012,KhesinMisiolekModin2019}) of the free Madelung system \eqref{eq:model} by
\begin{align}
	\smash{\mathcal{H}(\rho,\theta)
	\coloneqq
	\tfrac{1}{2m}\|\sqrt{\rho}\nabla\theta\|_{L^2(\mathbb{R}^d)}^2	
	+
	\tfrac{1}{8m}\mathcal{I}(\rho)
	=
	\tfrac{1}{2m}\bigl(\|\sqrt{\rho}\nabla\theta\|_{L^2(\mathbb{R}^d)}^2+\|\nabla\sqrt{\rho}\|_{L^2(\mathbb{R}^d)}^2\bigr)\,.}
	\label{eq:madelung_hamiltonian}
\end{align}
By definition \eqref{eq:velocity_field}, the first term in
\eqref{eq:madelung_hamiltonian} can be written as
$\smash{\frac{m}{2}\|\sqrt{\rho}\mathbf{v}\|^2_{L^2(\mathbb{R}^d)}}$~and,~thus,~represents the \emph{hydrodynamic kinetic energy}, while the second term
is the \emph{quantum kinetic contribution}~and~equals~$\smash{\frac{1}{8m}\mathcal{I}(\rho)}$. Under the
Madelung reconstruction \eqref{eq:madelung_reconstruction}, equivalently, one  has that
\begin{align}
	\smash{\mathcal{H}(\rho,\theta)
	=\tfrac{1}{2m}\|\nabla\Psi\|_{L^2(\mathbb{R}^d)}^2\,.}
	\label{eq:hamiltonian_schroedinger_energy}
\end{align}
Note that equation~\eqref{eq:hamiltonian_schroedinger_energy} is the \emph{free
Schrödinger energy} expressed in Madelung variables
(\textit{cf}.~\cite{vonRenesse2012,KhesinMisiolekModin2019}).~Its
conservation under the free Schrödinger evolution is standard
(\textit{cf}.~\cite[Chap.~2]{Cazenave2003};~see~also~\cite[Chap.~2]{Tao2006}).

In order to compute the first variations, we regard
\eqref{eq:madelung_hamiltonian} on the cone of smooth positive
densities,~not necessarily restricted to unit mass. For
every $\eta,\xi\hspace{-0.15em}\in\hspace{-0.15em} C_{\mathrm{c}}^\infty(\mathbb{R}^d)$, differentiation and~\mbox{integration~by~parts}~yield\vspace{-4.5mm}
\begin{subequations}
\label{eq:hamiltonian_first_variations}
\begin{align} 
	\tfrac{\mathrm{D}}{\mathrm{D}\varepsilon}
	\mathcal{H}(\rho+\varepsilon\eta,\theta)
	|_{\varepsilon=0}
	&=(\tfrac{1}{2m}|\nabla\theta|^2
		+
		Q_{\mathrm{B}}(\rho),\eta)_{L^2(\mathbb{R}^d)}\,,
	\label{eq:hamiltonian_density_variation}
	\\ 
	\tfrac{\mathrm{D}}{\mathrm{D}\varepsilon}
	\mathcal{H}(\rho,\theta+\varepsilon\xi)
    |_{\varepsilon=0}
	&=
	-\tfrac{1}{m}
	(\nabla\cdot(\rho\nabla\theta),\xi)_{L^2(\mathbb{R}^d)}\,.
	\label{eq:hamiltonian_phase_variation}
\end{align}
\end{subequations}
Consequently, the free Madelung system \eqref{eq:model} assumes the
canonical Hamiltonian form\vspace{-0.5mm}
\begin{subequations}
\label{eq:canonical_hamiltonian_system}
\begin{alignat}{2}
	\partial_t\rho
	&=
	\smash{\tfrac{\delta\mathcal{H}}{\delta\theta}}
	&&\quad\text{ in }\mathbb{R}^d\times(0,T)\,,
	\label{eq:canonical_density_equation}
	\\
	\partial_t\theta
	&=
	-\smash{\tfrac{\delta\mathcal{H}}{\delta\rho}}
	&&\quad\text{ in }\mathbb{R}^d\times(0,T)\,.
	\label{eq:canonical_phase_equation}\\[-6mm]\notag
\end{alignat}
\end{subequations}
Assuming sufficient temporal regularity, global integrability, and
decay at infinity to justify the preceding variational identities and
the Hamiltonian chain rule, the canonical Hamiltonian structure
\eqref{eq:canonical_hamiltonian_system} yields\vspace{-0.5mm}
\begin{alignat}{2}
	\tfrac{\mathrm{D}}{\mathrm{D}t}
	\mathcal{H}(\rho(\cdot,t),\theta(\cdot,t))
    =
	(\tfrac{\delta\mathcal{H}}{\delta\rho},\partial_t\rho)_{L^2(\mathbb{R}^d)}
	+
	(\tfrac{\delta\mathcal{H}}{\delta\theta},\partial_t\theta)_{L^2(\mathbb{R}^d)}
	=0\quad \text{ for all }t\in (0,T) \,.
	\label{eq:hamiltonian_conservation_rate}\\[-6mm]\notag
\end{alignat}
Integrating \eqref{eq:hamiltonian_conservation_rate} with respect to $t\in (0,T)$ yields\vspace{-0.5mm}
\begin{align}
	\mathcal{H}(\rho(\cdot,t),\theta(\cdot,t))
	=
	\mathcal{H}(\rho_0,\theta_0)
	\quad\text{ for all }t\in[0,T]\,.
	\label{eq:hamiltonian_conservation}\\[-6mm]\notag
\end{align} 

\hspace{-0.5mm}Since \hspace{-0.1mm}the \hspace{-0.1mm}Hamiltonian \hspace{-0.1mm}\eqref{eq:madelung_hamiltonian}, \hspace{-0.1mm}the \hspace{-0.1mm}hydrodynamic \hspace{-0.1mm}velocity
\hspace{-0.1mm}field \hspace{-0.1mm}\eqref{eq:velocity_field}, \hspace{-0.1mm}and \hspace{-0.1mm}the \hspace{-0.1mm}free \hspace{-0.1mm}Madelung~\hspace{-0.1mm}\mbox{system}~\hspace{-0.1mm}\eqref{eq:model} depend on 
$\theta$ only through its derivatives, if
$(\rho,\theta)$ is a solution,  so is $(\rho,\theta+C)$~for~all~${C\in\mathbb{R}}$.~Consequently, replacing $\theta_0$ by
$\theta_0+C$ leaves the hydrodynamic velocity field \eqref{eq:velocity_field}, the characteristic flow~\eqref{eq:characteristic_flow},~and the induced
probability-measure transport
\eqref{eq:deterministic_pushforward} unchanged. The phase-control
variable is, thus, naturally defined modulo additive constants,
and a normalization may~be~\mbox{imposed}~to~\mbox{select}~a~unique~\mbox{representative}.

Choosing the initial phase function $\theta_0\colon\hspace{-0.1em} \mathbb{R}^d\hspace{-0.1em}\to\hspace{-0.1em} \mathbb{R}$ fixes the initial velocity field
${\mathbf{v}(\cdot,0)\hspace{-0.1em}=\hspace{-0.1em}\frac{1}{m}\nabla\theta_0\colon \hspace{-0.1em}\mathbb{R}^d\hspace{-0.1em}\to \hspace{-0.1em}\mathbb{R}^d}$, whereas the subsequent
velocity field and characteristic flow are determined by the coupled
Hamiltonian evolution rather than prescribed independently. Therefore, the
admissible generative transports are constrained by the Hamiltonian free Madelung dynamics, including
the Fisher-information contribution. The
conservation law \eqref{eq:hamiltonian_conservation} reflects the
conservative rather than dissipative character of the evolution,
while its dispersive character is inherited from the 
free
Schrödinger~dynamics~(\textit{cf}.~\mbox{\cite[Chap.~2]{Tao2006}};~see~also~\cite{CarlesDanchinSaut2012}).\vspace{-1mm}\enlargethispage{2.5mm}

\subsection{PDE-constrained phase identification}
\label{subsec:pde_constrained_phase_identification} 

\hspace{4mm}The \hspace{-0.1mm}exact \hspace{-0.1mm}phase-control \hspace{-0.1mm}problem
\hspace{-0.1mm}\eqref{eq:phase_control_problem} \hspace{-0.1mm}need \hspace{-0.1mm}not \hspace{-0.1mm}be \hspace{-0.1mm}solvable \hspace{-0.1mm}for \hspace{-0.1mm}an \hspace{-0.1mm}arbitrary
\hspace{-0.1mm}target~\hspace{-0.1mm}\mbox{probability}~\hspace{-0.1mm}\mbox{measure} $\mu_{\star}\in \mathcal{P}(\mathbb{R}^d)$. This motivates a terminal-time
phase-identification problem in which the mismatch between the
terminal-time transported measure $\mu^{\theta_0}_T\in \mathcal{P}(\mathbb{R}^d)$ and the target measure~${\mu_{\star}\in \mathcal{P}(\mathbb{R}^d)}$~is~\mbox{minimized}. Let
$\Theta_{\mathrm{ad}}$ be a prescribed class of normalized initial
phase functions for which
Assumption~\ref{ass:classical_solution_setting} is satisfied. The
normalization removes the additive-constant ambiguity identified in
Subsection~\ref{subsec:hamiltonian_structure}. Moreover, let $\mathfrak{D}^2
	\colon
	\mathcal{P}(\mathbb{R}^d)\times
	\mathcal{P}(\mathbb{R}^d)
	\to[0,+\infty]$ be a prescribed \emph{discrepancy functional} satisfying, for all $\mu,\nu \in \mathcal{P}(\mathbb{R}^d)$,\vspace{-0.5mm}
\begin{align}
	\mathfrak{D}^2(\mu,\nu)=0
	\qquad\Longleftrightarrow\qquad
	\mu=\nu\quad \text{ in }\mathcal{P}(\mathbb{R}^d)\,,
	\label{eq:probability_measure_discrepancy}\\[-6mm]\notag
\end{align} 
and let the \emph{terminal-time control-to-state map} $\mathcal{S}_T\colon \Theta_{\mathrm{ad}}\to \mathcal{P}(\mathbb{R}^d)$,  for every $\theta_0\in\Theta_{\mathrm{ad}}$, be defined by\vspace{-0.5mm}
\begin{align}
	\mathcal{S}_T(\theta_0)
	\coloneqq
	\mu_T^{\theta_0}
	=
	\rho^{\theta_0}(\cdot,T)\,\mathrm{d}x
	=
	[X^{\theta_0}(T;\cdot)]_{\#}\mu_0
    \quad\text{ in }\mathcal{P}(\mathbb{R}^d)\,.
	\label{eq:terminal_phase_to_state_map}\\[-6mm]\notag
\end{align}
Then, the corresponding \emph{reduced PDE-constrained phase identification problem} consists of minimizing the \emph{reduced cost functional} $\mathcal{J}\colon \Theta_{\mathrm{ad}}\to [0,+\infty]$, defined for every  $\theta_0\in \Theta_{\mathrm{ad}}$ by\vspace{-0.5mm}
\begin{align}\label{eq:infinite_dimensional_phase_identification}
    \mathcal{J}(\theta_0)
	&\coloneqq
	\mathfrak{D}^2(
		\mathcal{S}_T(\theta_0),\mu_\star
	)\,.\\[-6mm]\notag
\end{align} 
Here, $\theta_0\in \Theta_{\mathrm{ad}}$ is the control variable, $(\rho^{\theta_0},
\theta^{\theta_0})\colon \mathbb{R}^d\times [0,T]\to \mathbb{R}_{>0}\times\mathbb{R}$ is the state variable, and the reduced cost functional \eqref{eq:infinite_dimensional_phase_identification}  depends
 only on the terminal probability measure
$\rho^{\theta_0}(\cdot,T)\,\mathrm{d}x\in \mathcal{P}(\mathbb{R}^d)$, and not on the terminal phase function $\theta^{\theta_0}(\cdot,T)\colon\mathbb{R}^d\to\mathbb{R}$. In particular,~by~equivalence~\eqref{eq:probability_measure_discrepancy}, an admissible initial phase function $\theta_0\in \Theta_{\mathrm{ad}}$ solves the exact
phase-control problem \eqref{eq:phase_control_problem}~if~and~only~if~$\mathcal{J}(\theta_0)=0$.

For computational purposes, let
$\mathcal A\subseteq \mathbb{R}^p$, $p\in \mathbb{N}$, be a parameter set and let
$\{\theta_0^\alpha\}_{\alpha\in\mathcal A}\subseteq\Theta_{\mathrm{ad}}$ be a finite-dimensional family of
normalized initial phase functions. Then, the \emph{reduced parameter-identification problem} consists of minimizing the \emph{finite-dimensional cost functional} $\mathcal{J}_{\mathcal{A}}\colon \mathcal{A}\to [0,+\infty]$, defined for every $\alpha\in \mathcal{A}$ by\vspace{-0.5mm}
\begin{align}
    J_{\mathcal{A}}(\alpha)\coloneqq \mathcal{J}(\theta_0^{\alpha})\,. 
	\label{eq:finite_dimensional_phase_identification}\\[-6mm]\notag
\end{align}
No general existence result for minimizers of
\eqref{eq:infinite_dimensional_phase_identification} or differentiability
result for the terminal-time control-to-state map
\eqref{eq:terminal_phase_to_state_map} is asserted here. For the
Gaussian reference density and quadratic phase family considered next, however, this map is explicit, permitting an analysis of intrinsic quantum spreading and an exact characterization of isotropic Gaussian reachability.\newpage

\section{Gaussian Wave-Packet Dynamics and
Phase-Controlled Reachability}
\label{sec:exact_gaussian}

\hspace{5mm}We specialize the free Madelung problem
\eqref{eq:model}--\eqref{eq:initial_conditions} to isotropic Gaussian
reference densities and quadratic initial phase functions. In this class, the
PDE dynamics reduce~to~\mbox{finite-dimensional} evolution equations for the
mean, width, and phase coefficients, and the terminal-time control-to-state map \eqref{eq:terminal_phase_to_state_map} can be evaluated explicitly.
First, we derive the resulting Gaussian dynamics and intrinsic quantum
spreading in arbitrary dimension and, subsequently, characterize exact terminal-time reachability within this
isotropic Gaussian class and obtain the
corresponding explicit sampling map.

\subsection{Isotropic Gaussian wave packets and intrinsic quantum spreading}
\label{subsec:gaussian_wave_packet_dynamics}

\hspace{5mm}The free propagation of 
Gaussian wave packets and the
effect of a quadratic initial phase function on their width are classical
(\textit{cf}.~\cite{RobinettDoncheskiBassett2005,Beach2009}).
We formulate these dynamics in Madelung variables and explicitly solve
the corresponding characteristic-flow problem
\eqref{eq:characteristic_flow}, thereby making the dependence on the
initial phase function explicit for the reachability analysis below. For an
optimal-transport~interpretation of a particular freely propagating
Gaussian state, see also \cite{Lessel2025}.

\begin{proposition}[Isotropic Gaussian wave-packet dynamics]
\label{prop:isotropic_gaussian_dynamics}
Let $\sigma_0>0$, $a_0,c_0\in\mathbb{R}$, 
$\boldsymbol{\mu}_0,\mathbf{u}_0\in\mathbb{R}^d$,
and let the isotropic
Gaussian initial probability density $\rho_0\colon \mathbb{R}^d\to \mathbb{R}_{>0}$ and the quadratic initial phase function $\theta_0\colon \mathbb{R}^d\to \mathbb{R}$, for every $x\in \mathbb{R}^d$, respectively, be defined by\vspace{-0.5mm}
\begin{subequations}\label{eq:gaussian_initial} 
\begin{alignat}{2}
	\rho_0(x)
	&\coloneqq
	\tfrac{1}{(2\pi \sigma_0^2)^{d/2}}
	\exp(
		-\tfrac{|x-\boldsymbol{\mu}_0|^2}{2\sigma_0^2}
	)\,,\label{eq:gaussian_initial_density}\\
    \theta_0(x)
	&\coloneqq
	\tfrac{m a_0}{2\sigma_0}|x-\boldsymbol{\mu}_0|^2
	+
	m\mathbf{u}_0\cdot(x-\boldsymbol{\mu}_0)
	+c_0\,.
	\label{eq:gaussian_initial_phase}
\end{alignat}
\end{subequations}
Moreover, let the \emph{center function} $\boldsymbol{\mu}\colon [0,T]\to \mathbb{R}^d$, the \emph{width function} $\sigma\colon [0,T]\to \mathbb{R}_{>0}$, and the \emph{phase offset function} $\gamma\colon [0,T]\to \mathbb{R}$, for every $t\in [0,T]$, respectively, be defined by\vspace{-0.5mm}
\begin{subequations}
\label{eq:gaussian_center_and_width}
\begin{align}
	\boldsymbol{\mu}(t)
	&\coloneqq
	\boldsymbol{\mu}_0+t\mathbf{u}_0\,,
	\label{eq:gaussian_center}
	\\[-0.5mm]
	\sigma(t)
	&\coloneqq
	\bigl(
		(\sigma_0+a_0t)^2
		+\tfrac{t^2}{4m^2\sigma_0^2}\bigr)^{1/2}\,,\label{eq:gaussian_width}\\[-0.5mm]
    \gamma(t)
	&\coloneqq
	c_0+\tfrac{m}{2}|\mathbf{u}_0|^2t
	-\tfrac{d}{4m}\textstyle \int_0^t\tfrac{\mathrm{D}\tau}{\sigma(\tau)^2}\,.\label{eq:gaussian_scalar_phase} 
\end{align}
\end{subequations} 
Then, the isotropic Gaussian probability density $\rho\colon \mathbb{R}^d\times [0,T]\to \mathbb{R}_{>0}$ and the phase function $\theta\colon \mathbb{R}^d\times [0,T]\to \mathbb{R}$, for every $(x,t)\in \mathbb{R}^d\times [0,T]$, respectively, defined by\vspace{-0.5mm}
\begin{subequations}
\label{eq:gaussian_madelung_solution}
\begin{align}
	\rho(x,t)
	&\coloneqq
	\tfrac{1}{(2\pi \sigma(t)^2)^{d/2}}
	\exp(
		-\tfrac{|x-\boldsymbol{\mu}(t)|^2}{2\sigma(t)^2}
	)\,,
	\label{eq:gaussian_density}
	\\
	\theta(x,t)
	&\coloneqq
	\smash{\tfrac{m\dot{\sigma}(t)}{2\sigma(t)}}
	|x-\boldsymbol{\mu}(t)|^2
	+
	m\mathbf{u}_0\cdot(x-\boldsymbol{\mu}(t))
	+\gamma(t)\,,
	\label{eq:gaussian_phase}
\end{align}
\end{subequations}
constitute a positive classical solution of the free Madelung
problem
\eqref{eq:model}--\eqref{eq:initial_conditions}. Moreover, 
the associated hydrodynamic velocity field $\mathbf{v}\colon \mathbb{R}^d\times [0,T]\to \mathbb{R}^d$, for every $(x,t)\in \mathbb{R}^d\times [0,T]$, is given via
\begin{align}
	\mathbf{v}(x,t)
	=
	\mathbf{u}_0
	+\smash{\tfrac{\dot{\sigma}(t)}{\sigma(t)}}
	(x-\boldsymbol{\mu}(t))\,,
	\label{eq:gaussian_velocity_field}
\end{align}
and, consequently, the corresponding characteristic flow $X\colon [0,T]\times \mathbb{R}^d\to \mathbb{R}^d$, for every $(t,x)\in [0,T]\times \mathbb{R}^d$, is given via\vspace{-0.5mm} 
\begin{alignat}{2}
	X(t;x)=
	\boldsymbol{\mu}(t)
	+\smash{\tfrac{\sigma(t)}{\sigma_0}}(x-\boldsymbol{\mu}_0)\,,
	\label{eq:gaussian_characteristic_flow}\\[-6mm]\notag
\end{alignat}
and satisfies
\begin{align}
	[X(t;\cdot)]_{\#}(\rho_0\,\mathrm{d}x)
	&=
	\rho(\cdot,t)\,\mathrm{d}x
	\quad\text{ in }\mathcal{P}(\mathbb{R}^d)
	\quad\text{ for all }t\in[0,T]\,.
	\label{eq:gaussian_pushforward}
\end{align}
In particular, if
$X_0\sim\mathcal{N}(\boldsymbol{\mu}_0,\sigma_0^2I_d)$, then
\begin{align}
	X(t;X_0)
	\sim
	\mathcal{N}(\boldsymbol{\mu}(t),\sigma(t)^2I_d)
	\quad\text{ for all }t\in[0,T]\,,
	\label{eq:gaussian_sample_distribution}
\end{align}
where $\mathcal{N}(\boldsymbol{\mu},\sigma^2 I_d)$ denotes the $d$-dimensional isotropic 
Gaussian distribution with mean $\boldsymbol{\mu}\in\mathbb{R}^d$ and covariance 
matrix $\sigma^2 I_d$ with $\sigma>0$.
\end{proposition}

\begin{proof}
From the definition of the isotropic Gaussian probability density and the phase function \eqref{eq:gaussian_madelung_solution}, for every $(x,t)
	\in\mathbb{R}^d\times[0,T]$, we deduce that
\begin{subequations}
\label{eq:gaussian_density_derivatives}
\begin{alignat}{2}
	\nabla\log\rho(x,t)
	&=
	-\tfrac{x-\boldsymbol{\mu}(t)}{\sigma(t)^2}\,,
	\label{eq:gaussian_logarithmic_gradient}
	\\
	\partial_t\log\rho(x,t)
	&=
	-d \tfrac{\dot{\sigma}(t)}{\sigma(t)}
	+\tfrac{\mathbf{u}_0\cdot(x-\boldsymbol{\mu}(t))}{\sigma(t)^2}
	+\tfrac{\dot{\sigma}(t)}{\sigma(t)}\tfrac{|x-\boldsymbol{\mu}(t)|^2}{\sigma(t)^2}\,,\\
    \tfrac{1}{m}\nabla\theta(x,t)
	&=
	\mathbf{u}_0+\tfrac{\dot{\sigma}(t)}{\sigma(t)}(x-\boldsymbol{\mu}(t))\,.
	\label{eq:gaussian_logarithmic_time_derivative}
\end{alignat}
\end{subequations} 
Since
$\operatorname{div}\mathbf{v}(\cdot,t)=d \tfrac{\dot{\sigma}(t)}{\sigma(t)}$ in $\mathbb{R}^d$ for all $t\in [0,T]$, combining
\eqref{eq:gaussian_density_derivatives} verifies the continuity equation
\eqref{eq:model_density}.

For the isotropic Gaussian probability density \eqref{eq:gaussian_density}, the Bohm quantum potential $Q_{\mathrm{B}}(\rho)\colon \mathbb{R}^d\times[0,T]\to \mathbb{R}$, defined by \eqref{eq:bohm_quantum_potential}, for every $(x,t)\in \mathbb{R}^d\times [0,T]$, assumes the form 
\begin{align}
	\smash{Q_{\mathrm{B}}(\rho)(x,t)
	=
	\tfrac{d}{4m \sigma(t)^2}
	-\tfrac{|x-\boldsymbol{\mu}(t)|^2}{8m \sigma(t)^4}}\,.
	\label{eq:gaussian_bohm_quantum_potential}
\end{align}
Due to the definition of the quadratic phase function \eqref{eq:gaussian_phase}, \eqref{eq:gaussian_bohm_quantum_potential}, $\ddot{\sigma}=\frac{1}{4m^2\sigma^3}$ in $[0,T]$, $\sigma(0)=\sigma_0$, $\dot{\sigma}(0)=a_0$, and $\dot{\gamma}
	=
	\frac{m}{2}|\mathbf{u}_0|^2
	-\frac{d}{4m \sigma^2}$ in $[0,T]$, for every $(x,t)\in \mathbb{R}^d\times [0,T]$, we find that  
\begin{align}\notag
	\partial_t\theta(x,t)
	+\tfrac{1}{2m}|\nabla\theta(x,t)|^2
	+Q_{\mathrm{B}}(\rho)(x,t)
	&=
	\bigl(
		\tfrac{m}{2}\tfrac{\ddot{\sigma}(t)}{\sigma(t)}
		-\tfrac{1}{8m \sigma(t)^4}
	\bigr)|x-\boldsymbol{\mu}(t)|^2 +
	\dot{\gamma}(t)
	-\tfrac{m}{2}|\mathbf{u}_0|^2
	+\tfrac{d}{4m \sigma(t)^2}
    \\&=0\,,
	\label{eq:gaussian_phase_equation_residual}
\end{align}
which \hspace{-0.1mm}verifies \hspace{-0.1mm}the \hspace{-0.1mm}quantum \hspace{-0.1mm}Hamilton--Jacobi
\hspace{-0.1mm}equation \hspace{-0.1mm}\eqref{eq:model_phase}. 
\hspace{-0.5mm}The \hspace{-0.1mm}initial \hspace{-0.1mm}conditions~\hspace{-0.1mm}follow~\hspace{-0.1mm}from~\hspace{-0.1mm}\mbox{\eqref{eq:gaussian_initial}--\eqref{eq:gaussian_madelung_solution}}.

Finally, differentiating the definition of characteristic flow \eqref{eq:gaussian_characteristic_flow} in
time using \eqref{eq:gaussian_velocity_field} shows~that~it~solves
the characteristic-flow problem \eqref{eq:characteristic_flow}.
The pushforward identity follows directly~from~\eqref{eq:gaussian_density},~\eqref{eq:gaussian_characteristic_flow},
and the affine transformation rule for Gaussian random variables (\textit{cf}.~\mbox{\cite[Thm.~3.3.3, p.~32]{Tong1990}}).
\end{proof}

If we choose
$a_0=0$, $\mathbf{u}_0=0$, and $c_0=0$ in Proposition \ref{prop:isotropic_gaussian_dynamics}, we obtain
$\theta_0\equiv 0$, $\boldsymbol{\mu}\equiv \boldsymbol{\mu}_0$, and 
\begin{subequations}
\label{eq:intrinsic_gaussian_spreading}
\begin{alignat}{2}
	\sigma(t)
	&=
	\smash{\sigma_0\bigl(
		1+\tfrac{t^2}{4m^2\sigma_0^4}
	\bigr)^{1/2}}
	&&\quad\text{ for all }t\in[0,T]\,,
	\label{eq:intrinsic_gaussian_width}
	\\
	\theta(x,t)
	&=
	\tfrac{mt}{2(t^2+4m^2\sigma_0^4)}
	|x-\boldsymbol{\mu}_0|^2
	-\tfrac{d}{2}
	\arctan(\tfrac{t}{2m\sigma_0^2})
	&&\quad\text{ for all }(x,t)
	\in\mathbb{R}^d\times[0,T]\,,
	\label{eq:intrinsic_gaussian_phase}
	\\
	X(t;x)
	&=
	\boldsymbol{\mu}_0
	+\tfrac{\sigma(t)}{\sigma_0}(x-\boldsymbol{\mu}_0)
	&&\quad\text{ for all }(t,x)
	\in[0,T]\times\mathbb{R}^d\,.
	\label{eq:intrinsic_gaussian_flow}
\end{alignat}
\end{subequations}
Therefore, \hspace{-0.1mm}even \hspace{-0.1mm}for \hspace{-0.1mm}a \hspace{-0.1mm}trivial \hspace{-0.1mm}initial \hspace{-0.1mm}phase \hspace{-0.1mm}function \hspace{-0.1mm}and, \hspace{-0.1mm}thus, \hspace{-0.1mm}a \hspace{-0.1mm}vanishing \hspace{-0.1mm}initial
\hspace{-0.1mm}hydrodynamic~\hspace{-0.1mm}velocity~\hspace{-0.1mm}field, the free Madelung dynamics are
nonstationary: the Bohm quantum potential induces the radial dilation
\eqref{eq:intrinsic_gaussian_flow}, with $\sigma(t)>\sigma_0$ for all $t>0$.
This is the intrinsic quantum spreading of the Gaussian~wave~packet.

More \hspace{-0.1mm}generally, \hspace{-0.1mm}the \hspace{-0.1mm}definitions \hspace{-0.1mm}\eqref{eq:gaussian_center} \hspace{-0.1mm}and
\hspace{-0.1mm}\eqref{eq:gaussian_width} \hspace{-0.1mm}show \hspace{-0.1mm}that \hspace{-0.1mm}$\mathbf{u}_0$ \hspace{-0.1mm}determines \hspace{-0.1mm}the
\hspace{-0.1mm}terminal-time~\hspace{-0.1mm}mean~\hspace{-0.1mm}$\boldsymbol{\mu}(T)$, whereas $a_0$ determines the terminal-time
standard deviation $\sigma(T)$. Consequently, the restriction~of~the
terminal-time control-to-state map
\eqref{eq:terminal_phase_to_state_map} to this finite-dimensional
quadratic-phase~\mbox{family}~is~\mbox{explicit}. This reduces exact reachability of
isotropic Gaussian target measures to algebraic conditions, which are
analyzed~in~the~next~subsection.

\subsection{Exact Gaussian reachability by phase control}
\label{sec:exact_gaussian_reachability}

\hspace{5mm}The explicit Gaussian dynamics derived in
Proposition~\ref{prop:isotropic_gaussian_dynamics} reduce the exact
phase-control problem \eqref{eq:phase_control_problem} to algebraic
conditions on the linear and quadratic coefficients of the initial
phase. We retain the reference probability measure
$\mu_0\coloneqq\rho_0\,\mathrm{d}x
=\mathcal{N}(\boldsymbol{\mu}_0,\sigma_0^2I_d)$ from
\eqref{eq:gaussian_initial_density} and consider the isotropic Gaussian
target density $\rho_\star\colon \mathbb{R}^d\to \mathbb{R}_{>0}$, defined for every $x\in \mathbb{R}^d$ by
\begin{align}
	\rho_\star(x)
	\coloneqq
	\tfrac{1}{(2\pi\sigma_\star^2)^{d/2}}
	\exp(
		-\tfrac{|x-\boldsymbol{\mu}_\star|^2}{2\sigma_\star^2})\,,
	\label{eq:gaussian_target_density}
\end{align}
where $\boldsymbol{\mu}_\star\in\mathbb{R}^d$ and $\sigma_\star>0$. The
corresponding target probability measure is defined by
\begin{align}
	\mu_\star
	\coloneqq
	\rho_\star\,\mathrm{d}x
	=
	\mathcal{N}(\boldsymbol{\mu}_\star,\sigma_\star^2I_d)
	\quad\text{ in }\mathcal{P}(\mathbb{R}^d)\,.
	\label{eq:gaussian_target_measure}
\end{align}

\begin{theorem}[Exact reachability of isotropic Gaussian measures]
\label{thm:exact_gaussian_reachability}
There exists a quadratic initial phase function $\theta_0\colon \mathbb{R}^d\to \mathbb{R}$ of the form
\eqref{eq:gaussian_initial_phase} such that, for some $\boldsymbol{\mu}_{\star}\in \mathbb{R}^d$ and $\sigma_{\star}>0$, there holds\vspace{-0.5mm}
\begin{alignat}{2}
	[X^{\theta_0}(T;\cdot)]_{\#}
	(\rho_0\,\mathrm{d}x)
	&=
	\mu_\star
    = \mathcal{N}(\boldsymbol{\mu}_*,\sigma_*^2I_d)
	&&\quad\text{in }\mathcal{P}(\mathbb{R}^d)\,,
	\label{eq:exact_gaussian_reachability_condition}\\[-6mm]\notag
\end{alignat}
if and only if\vspace{-0.5mm}
\begin{align}
	\sigma_\star
	\geq
	\tfrac{T}{2m\sigma_0}\,.
	\label{eq:gaussian_reachability_bound}\\[-6mm]\notag
\end{align}
If \eqref{eq:gaussian_reachability_bound} holds, then the
quadratic initial phase functions of the form
\eqref{eq:gaussian_initial_phase} satisfying
\eqref{eq:exact_gaussian_reachability_condition}~are~precisely 
$\smash{\theta_{0,\star}^{\pm}}\colon \mathbb{R}^d\to \mathbb{R}$, for every $x\in \mathbb{R}^d$, given via
\begin{align}
	\theta_{0,\star}^{\pm}(x)
	\coloneqq
	\smash{\tfrac{m a_{0,\star}^{\pm}}{2\sigma_0}}
	|x-\boldsymbol{\mu}_0|^2
	+
	m\mathbf{u}_{0,\star}\cdot(x-\boldsymbol{\mu}_0)
	+c_0\,,
	\label{eq:exact_gaussian_initial_phases}\\[-6.5mm]\notag
\end{align}
where $\mathbf{u}_{0,\star}
	\coloneqq
	\tfrac{1}{T}(\boldsymbol{\mu}_\star-\boldsymbol{\mu}_0)\in \mathbb{R}^d$, $\smash{a_{0,\star}^{\pm}}\coloneqq
	\tfrac{1}{T}(-\sigma_0
	\pm(\sigma_\star^2-\tfrac{T^2}{4m^2\sigma_0^2})^{\smash{\frac{1}{2}}})\in \mathbb{R}$, and $c_0\in\mathbb{R}$ is arbitrary. If inequality 
\eqref{eq:gaussian_reachability_bound} is strict, the two quadratic
coefficients $a_{0,\star}^{\pm}\in \mathbb{R}$ are distinct; at
equality,~they~coincide.\enlargethispage{2.5mm}
\end{theorem}

\begin{proof}
For a quadratic initial phase function $\theta_0\colon \mathbb{R}^d\to \mathbb{R}$ of the form
\eqref{eq:gaussian_initial_phase},
Proposition~\ref{prop:isotropic_gaussian_dynamics} yields
\begin{alignat}{2}
	[X^{\theta_0}(T;\cdot)]_{\#}
	(\rho_0\,\mathrm{d}x)
	&=
	\smash{\mathcal{N}\bigl(
		\boldsymbol{\mu}_0+T\mathbf{u}_0,
		\bigl(
		(\sigma_0+a_0T)^2
		+\tfrac{T^2}{4m^2\sigma_0^2}
		\bigr)I_d
	\bigr)}
	&&\quad\text{in }\mathcal{P}(\mathbb{R}^d)\,.
	\label{eq:terminal_gaussian_parameter_to_state_map}
\end{alignat}
Therefore,
equation \eqref{eq:exact_gaussian_reachability_condition} holds if and only if
the terminal-time mean and covariance,~respectively,~satisfy\vspace{-5mm}
\begin{subequations}
\label{eq:terminal_gaussian_matching_conditions}
\begin{align}
	\boldsymbol{\mu}_\star&=\boldsymbol{\mu}_0+T\mathbf{u}_0
	\,,
	\label{eq:terminal_gaussian_mean_matching}
	\\
	\sigma_\star^2
	&=(\sigma_0+a_0T)^2
	+\smash{\tfrac{T^2}{4m^2\sigma_0^2}}
	\,.
	\label{eq:terminal_gaussian_width_matching}\\[-6mm]\notag
\end{align}
\end{subequations}
Equation \eqref{eq:terminal_gaussian_mean_matching} uniquely determines
$\mathbf{u}_0=\mathbf{u}_{0,\star}$. Moreover,
\eqref{eq:terminal_gaussian_width_matching} admits a real solution
$a_0$ if and only if \eqref{eq:gaussian_reachability_bound} holds, in
which case its solutions are precisely
$a_0=a_{0,\star}^{\pm}$. The additive constant $c_0$ does not affect
the Born probability density, hydrodynamic velocity field, or characteristic flow. 
\end{proof}

Whenever
$\theta_{0,\star}^{\pm}\in\Theta_{\mathrm{ad}}$, the separation
property \eqref{eq:probability_measure_discrepancy} implies that $\mathcal{J}(\theta_{0,\star}^{\pm})
	=0$,  
so that each reachable phase function is a global minimizer of the reduced
PDE-constrained phase identification~problem~\eqref{eq:infinite_dimensional_phase_identification}. The constant
$c_0$ in \eqref{eq:exact_gaussian_initial_phases} may be selected to
satisfy the normalization imposed in $\Theta_{\mathrm{ad}}$.

\begin{corollary}[Closed-form Gaussian sampling map]
\label{cor:closed_form_gaussian_sampling}
Let inequality \eqref{eq:gaussian_reachability_bound} hold and let the affine transport map $X_\star\colon \mathbb{R}^d\to \mathbb{R}^d$, for every $x\in \mathbb{R}^d$, be defined by\vspace{-0.5mm}
\begin{align}
	X_\star(x)
	\coloneqq
	\boldsymbol{\mu}_\star
	+\tfrac{\sigma_\star}{\sigma_0}
	(x-\boldsymbol{\mu}_0)\,.
	\label{eq:exact_gaussian_sampling_map}\\[-6mm]\notag
\end{align}
Then, for either quadratic branch
$\theta_{0,\star}^{\pm}\colon \mathbb{R}^d\to\mathbb{R}$, defined by 
\eqref{eq:exact_gaussian_initial_phases}, the terminal-time transport~map $X^{\smash{\theta_{0,\star}^{\pm}}}(T;\cdot)\colon\mathbb{R}^d\to \mathbb{R}^d$ of the 
characteristic flow $X^{\smash{\theta_{0,\star}^{\pm}}}\colon [0,T]\times \mathbb{R}^d\to \mathbb{R}^d$,~for~every~$x\in \mathbb{R}^d$,~satisfies\vspace{-0.5mm}
\begin{align}
	X^{\theta_{0,\star}^{\pm}}(T;x)
	=
	X_\star(x)\,,
	\label{eq:exact_gaussian_terminal_flow}\\[-6mm]\notag
\end{align}
and if
$X_0\colon\hspace{-0.1em}\Xi\hspace{-0.1em}\to\hspace{-0.1em}\mathbb{R}^d$ is a random variable on a probability space
$(\Xi,\mathcal F,\mathbb{P})$ satisfying
${[X_0]_{\#}\mathbb{P}\hspace{-0.1em}=\hspace{-0.1em}\rho_0\,\mathrm{d}x}$, then\vspace{-0.5mm}
\begin{alignat}{2}
	[X_\star(X_0)]_{\#}\mathbb{P}
	&=
	\mu_\star
	&&\quad\text{in }\mathcal{P}(\mathbb{R}^d)\,.
	\label{eq:exact_gaussian_generated_distribution}\\[-6mm]\notag
\end{alignat}
\end{corollary}

\begin{proof}
By Theorem~\ref{thm:exact_gaussian_reachability}, we have that
$\boldsymbol{\mu}(T)=\boldsymbol{\mu}_\star$ and $\sigma(T)=\sigma_\star$. As a consequence,
\eqref{eq:gaussian_characteristic_flow} yields
\eqref{eq:exact_gaussian_terminal_flow}, while
\eqref{eq:exact_gaussian_generated_distribution} follows from
Proposition~\ref{prop:characteristic_flow_representation}.
\end{proof}

\begin{remark}[Nonuniqueness and scope]
\label{rem:gaussian_reachability_scope}

If \eqref{eq:gaussian_reachability_bound} is strict, the two
inequivalent quadratic branches
$\theta_{0,\star}^{\pm}\colon \mathbb{R}^d\to \mathbb{R}$, modulo additive constants, generate the same
terminal-time transport map $X^{\smash{\theta_{0,\star}^{\pm}}}(T;\cdot)=X_{\star}\colon \mathbb{R}^d\to \mathbb{R}^d$, although
their intermediate width functions and characteristic flows generally
differ. In other words, terminal-time-density matching does not uniquely determine
the initial phase function.

For $d=1$, Theorem~\ref{thm:exact_gaussian_reachability} applies to
every nondegenerate Gaussian target measure. For $d>1$, it concerns
the isotropic Gaussian family; anisotropic Gaussian targets would
require matrix-valued quadratic phase coefficients. Moreover,
\eqref{eq:gaussian_reachability_bound} characterizes reachability
within the quadratic-phase-function family considered here and is not asserted
to be an obstruction for arbitrary~admissible~initial~phase~functions.
\end{remark}

The exact construction above relies on the invariance of the Gaussian
family under~the~free~Schrödinger evolution. Next, we  leave this
invariant class and study how suitable initial phase functions realize
general nonlinear potential transport directions to first order over
short~time~\mbox{intervals}~$[0,T]$,~as~${T\to0}$.\newpage

\section{Local-in-Time Realization of Nonlinear Transport Maps}
\label{sec:local_transport_realization}

\hspace{5mm}The \hspace{-0.1mm}exact \hspace{-0.1mm}reachability \hspace{-0.1mm}result \hspace{-0.1mm}in
\hspace{-0.1mm}Section~\hspace{-0.1mm}\ref{sec:exact_gaussian} \hspace{-0.1mm}relies \hspace{-0.1mm}on \hspace{-0.1mm}the \hspace{-0.1mm}invariance \hspace{-0.1mm}of \hspace{-0.1mm}the
\hspace{-0.1mm}isotropic \hspace{-0.1mm}Gaussian \hspace{-0.1mm}family \hspace{-0.1mm}under the free Schrödinger evolution. We now
consider general nonlinear potential~\mbox{transport}~\mbox{directions}~in~$\mathbb{R}^d$. The result is local in time: an initial phase function is chosen
to reproduce a prescribed potential velocity field and the resulting
Madelung flow is compared with the corresponding first-order
transport~map.

Let
$\mathbf{v}_0\in C^1(\mathbb{R}^d;\mathbb{R}^d)$ be a potential velocity field
satisfying $\mathrm{D}\mathbf{v}_0\in L^\infty(\mathbb{R}^d;\mathbb{R}^{d\times d})$. More precisely, suppose that there exists
$\varphi_0\in C^2(\mathbb{R}^d)$ such that
$\mathbf{v}_0=\nabla\varphi_0$ in $\mathbb{R}^d$, and choose the
initial~phase~function~as\vspace{-0.5mm}
\begin{align}
	\theta_0
	\coloneqq
	m\varphi_0+c_0
	\in C^2(\mathbb{R}^d)\,, \quad \text{ for some }c_0\in\mathbb{R}\,.
	\label{eq:prescribed_initial_phase}\\[-6mm]\notag
\end{align}
By
definition \eqref{eq:velocity_field}, the initial condition
\eqref{eq:initial_phase}, and
\eqref{eq:prescribed_initial_phase}, we have that
\begin{align}
\smash{	\mathbf{v}^{\theta_0}(\cdot,0)
	=
	\tfrac{1}{m}\nabla\theta_0
	=
	\mathbf{v}_0
	\quad\text{ in }\mathbb{R}^d\,.}
	\label{eq:initial_velocity_matching}
\end{align}
For the fixed initial phase function \eqref{eq:prescribed_initial_phase},  the corresponding
\emph{first-order comparison map} $\widetilde{X}_T^{\theta_0}\colon \mathbb{R}^d\to \mathbb{R}^d$ is
defined as the forward-Euler approximation of the characteristic~flow~\eqref{eq:characteristic_flow}~at~$t\hspace{-0.15em}=\hspace{-0.15em}0$,~\hspace{-0.1mm}\textit{i.e.},~\hspace{-0.1mm}for~\hspace{-0.1mm}\mbox{every}~\hspace{-0.1mm}${x\hspace{-0.15em}\in\hspace{-0.15em}\mathbb{R}^d}$, we have that\vspace{-0.5mm}
\begin{align}
    \smash{\widetilde{X}_T^{\theta_0}(x)
	\coloneqq
	x+T\mathbf v^{\theta_0}(x,0)
	=
	x+\tfrac{T}{m}\nabla\theta_0(x)
	=
	x+T\mathbf v_0(x)\,.}
	\label{eq:first_order_transport_map}\\[-6mm]\notag
\end{align}  

Next, for an initial probability density $\rho_0\colon\mathbb{R}^d\to \mathbb{R}_{>0}$, 
let
$\mu_0\coloneqq\rho_0\,\mathrm{d}x
\in\mathcal{P}_2(\mathbb{R}^d)$ be the corresponding initial probability measure and define the corresponding
\emph{first-order comparison measure} by
\begin{align}
	\smash{\widetilde{\mu}_T^{\theta_0}
	\coloneqq
	[\widetilde{X}_T^{\theta_0}]_{\#}\mu_0\in \mathcal{P}_2(\mathbb{R}^d)\,.}
	\label{eq:first_order_target_measure}\\[-6mm]\notag
\end{align}
Here, for $p\in[1,+\infty)$, $\mathcal{P}_p(\mathbb{R}^d)
	\coloneqq \{
		\mu\in\mathcal{P}(\mathbb{R}^d)
		\mid
		\int_{\mathbb{R}^d}|x|^p\,\mathrm{d}\mu(x)<+\infty \}$ 
denotes the \emph{space of Borel probability measures with finite
$p$-th moment}. For
$\mu,\nu\in\mathcal{P}_p(\mathbb{R}^d)$, their
\emph{$p$-Wasserstein~distance}~is~defined~by\vspace{-1mm}
\begin{align}
	W_p(\mu,\nu)
	\coloneqq
	\inf_{\pi\in\Pi(\mu,\nu)}
	\left(
		\int_{\mathbb{R}^d\times\mathbb{R}^d}
		|x-y|^p\,\mathrm{d}\pi(x,y)
	\right)^{1/p}\,,
	\label{eq:wasserstein_distance}\\[-6mm]\notag
\end{align}
where $\Pi(\mu,\nu)$ denotes the set of couplings of $\mu$ and $\nu$
(see, \textit{e.g.}, \cite[Sec.~7.1]{AmbrosioGigliSavare2008}).

The quantum Hamilton--Jacobi equation
\eqref{eq:model_phase} determines the subsequent evolution of the
hydrodynamic velocity field 
from its prescribed initial
value  
in
\eqref{eq:initial_velocity_matching}. Indeed, taking the spatial
gradient~of~\eqref{eq:model_phase}, using that $\mathrm{D}\mathbf{v}^{\theta_0}=\frac{1}{m}\nabla^2 \theta^{\theta_0}$ is symmetric in $\mathbb{R}^d\times[0,T]$, and dividing by $m>0$ yields~the~\emph{material acceleration}
$\mathbf{a}^{\theta_0}\colon\mathbb{R}^d\times(0,T)\to\mathbb{R}^d$, defined by\vspace{-0.5mm}
\begin{align}
	\smash{\mathbf{a}^{\theta_0}
	\coloneqq
	\partial_t\mathbf{v}^{\theta_0}
	+[\mathrm{D}\mathbf{v}^{\theta_0}]\mathbf{v}^{\theta_0}
	=
	-\tfrac{1}{m}\nabla Q_{\mathrm B}(\rho^{\theta_0})
	\quad\text{ in }\mathbb{R}^d\times(0,T)\,.}
	\label{eq:material_acceleration}\\[-6mm]\notag
\end{align}
The following theorem gives an exact integral representation of 
$X^{\theta_0}(T;\cdot)-\widetilde X_T^{\theta_0}$
in terms of the material acceleration $\mathbf a^{\theta_0}$.
A uniform bound on $\mathbf a^{\theta_0}$, equivalently on
$\nabla Q_{\mathrm B}(\rho^{\theta_0})$, then yields~\mbox{quadratic-in-time}~bounds for this map error and, through the coupling
induced by $\mu_0$, for the corresponding $W_1$- and
$W_2$-distances.

\begin{theorem}[Second-order local transport realization]
\label{thm:second_order_transport_realization}
Let Assumption~\ref{ass:classical_solution_setting} be satisfied for
the initial probability density
$\rho_0\colon\mathbb{R}^d\to\mathbb{R}_{>0}$ and the initial phase
function
$\theta_0\colon\mathbb{R}^d\to\mathbb{R}$ defined by
\eqref{eq:prescribed_initial_phase}.
Moreover, set $\mu_0\coloneqq\rho_0\,\mathrm{d}x
\in\mathcal{P}_2(\mathbb{R}^d)$ and assume that\vspace{-0.5mm}
\begin{align}
	\smash{\mathbf{a}^{\theta_0}
    \in L^\infty(\mathbb{R}^d\times(0,T);\mathbb{R}^d)\,.}
	\label{eq:material_acceleration_bound}\\[-6mm]\notag
\end{align}
Then, for every $x\in\mathbb{R}^d$, the corresponding characteristic flow
$X^{\theta_0}\colon[0,T]\times\mathbb{R}^d\to\mathbb{R}^d$ satisfies\vspace{-0.5mm} 
\begin{align}
	X^{\theta_0}(T;x)-\widetilde{X}_T^{\theta_0}(x)
	=
	\int_0^T
	(T-t)\mathbf{a}^{\theta_0}(X^{\theta_0}(t;x),t)\,\mathrm{d}t \,,
	\label{eq:characteristic_second_order_identity}\\[-6mm]\notag
\end{align}
and, consequently,
\begin{align}
	\smash{\|X^{\theta_0}(T;\cdot)-\widetilde{X}_T^{\theta_0}\|_{L^\infty(\mathbb{R}^d;\mathbb{R}^d)}
	\leq
	\tfrac{T^2}{2}\|\mathbf{a}^{\theta_0}\|_{L^\infty(\mathbb{R}^d\times(0,T);\mathbb{R}^d)}\,.}
	\label{eq:uniform_second_order_map_estimate}
\end{align}
Moreover, setting $\smash{\mu_T^{\theta_0}
	\coloneqq
	\rho^{\theta_0}(\cdot,T)\,\mathrm{d}x
	=
	[X^{\theta_0}(T;\cdot)]_{\#}\mu_0
	\in\mathcal{P}_2(\mathbb{R}^d)}$, for every $p\in\{1,2\}$, we have that
\begin{align}
	\smash{W_p(\mu_T^{\theta_0},\widetilde{\mu}_T^{\theta_0})
	\leq
	\tfrac{T^2}{2}\|\mathbf{a}^{\theta_0}\|_{L^\infty(\mathbb{R}^d\times(0,T);\mathbb{R}^d)}\,.}
	\label{eq:wasserstein_second_order_estimate}\\[-6mm]\notag
\end{align}
\end{theorem}\pagebreak

\begin{proof}
To begin with, conditions~(\hyperlink{A.1}{A.1}) and
(\hyperlink{A.3}{A.3}) of
Assumption~\ref{ass:classical_solution_setting}, together with
\eqref{eq:velocity_field}, imply that $\mathbf v^{\theta_0} \in
	C^0([0,T];C^1_{\mathrm{loc}}
	(\mathbb{R}^d;\mathbb{R}^d))
	\cap
	C^1([0,T];C^0_{\mathrm{loc}}
	(\mathbb{R}^d;\mathbb{R}^d))$. Since condition~(\hyperlink{A.4}{A.4}) provides~the~global~characteristic flow, \hspace{-0.15mm}we \hspace{-0.15mm}find \hspace{-0.15mm}that, \hspace{-0.15mm}for \hspace{-0.15mm}every
\hspace{-0.15mm}$x\hspace{-0.15em}\in\hspace{-0.15em}\mathbb{R}^d$, \hspace{-0.15mm}the \hspace{-0.15mm}trajectory
\hspace{-0.15mm}$X^{\theta_0}(\cdot;x)\colon\hspace{-0.15em}[0,T]\hspace{-0.15em}\to\hspace{-0.15em}\mathbb{R}^d$ \hspace{-0.15mm}is \hspace{-0.15mm}twice~\hspace{-0.15mm}\mbox{continuously}~\hspace{-0.15mm}\mbox{differentiable}. Therefore, differentiating the characteristic-flow equation
\eqref{eq:characteristic_flow_equation} with respect to time and
using \eqref{eq:material_acceleration}, for every $(t,x)
	\in(0,T)\times\mathbb{R}^d$, we obtain\vspace{-0.5mm}
\begin{align}
    \begin{aligned}
	\partial_t^2X^{\theta_0}(t;x)
	&=
	\partial_t\mathbf{v}^{\theta_0}(X^{\theta_0}(t;x),t)
	+
	[\mathrm{D}\mathbf{v}^{\theta_0}(X^{\theta_0}(t;x),t)]\partial_tX^{\theta_0}(t;x)
    \\&=
	\partial_t\mathbf{v}^{\theta_0}(X^{\theta_0}(t;x),t)
	+
	[\mathrm{D}\mathbf{v}^{\theta_0}(X^{\theta_0}(t;x),t)]\mathbf{v}^{\theta_0}(X^{\theta_0}(t;x),t)\\&=
	\mathbf{a}^{\theta_0}(X^{\theta_0}(t;x),t)\,.
    \end{aligned}
	\label{eq:characteristic_acceleration_equation}\\[-6mm]\notag
\end{align}
In addition, \eqref{eq:initial_velocity_matching} gives $\partial_tX^{\theta_0}(0;\cdot)
	=
	\mathbf{v}^{\theta_0}(\cdot,0)
	=
	\mathbf{v}_0$ in $\mathbb{R}^d$. 
As a result, integrating \eqref{eq:characteristic_acceleration_equation} twice in
time, we arrive at\vspace{-0.5mm}
\begin{align}
	X^{\theta_0}(T;\cdot)
	=
	\operatorname{id}_{\mathbb{R}^d}+T\mathbf{v}_0
	+
	\int_0^T
	(T-t)\mathbf{a}^{\theta_0}(X^{\theta_0}(t;\cdot),t)\,\mathrm{d}t
	\quad\text{ in }\mathbb{R}^d\,,
	\label{eq:characteristic_integral_remainder}\\[-6mm]\notag
\end{align}
which together with  
\eqref{eq:first_order_transport_map} proves
\eqref{eq:characteristic_second_order_identity} and, using 
\eqref{eq:material_acceleration_bound}, also 
\eqref{eq:uniform_second_order_map_estimate}.

Since, by the assumption $\mathrm{D}\mathbf{v}_0\in L^\infty(\mathbb{R}^d;\mathbb{R}^d)$, the initial velocity field $\mathbf{v}_0\in C^1(\mathbb{R}^d;\mathbb{R}^d)$ is globally Lipschitz continuous, the first-order comparison map $\widetilde{X}_T^{\theta_0}\colon \mathbb{R}^d\to \mathbb{R}^d$ has at most
linear~growth, so that, using the uniform bound \eqref{eq:uniform_second_order_map_estimate}, from $\mu_0\in\mathcal{P}_2(\mathbb{R}^d)$, it follows that
$\mu_T^{\theta_0},\widetilde{\mu}_T^{\theta_0}\in\mathcal{P}_2(\mathbb{R}^d)$. 

Next, in order to verify the Wasserstein estimates \eqref{eq:wasserstein_second_order_estimate}, we define the probability measure\vspace{-0.5mm}
\begin{alignat}{2}
	\pi_T^{\theta_0}
	&\coloneqq
	[(X^{\theta_0}(T;\cdot),\widetilde{X}_T^{\theta_0})]_{\#}\mu_0
	&&\quad\text{in }
	\mathcal{P}(\mathbb{R}^d\times\mathbb{R}^d)\,,
	\label{eq:transport_comparison_coupling}\\[-6mm]\notag
\end{alignat}
where \hspace{-0.1mm}$(X^{\theta_0}(T;\cdot),\widetilde{X}_T^{\theta_0})\colon \hspace{-0.175em}\mathbb{R}^d\hspace{-0.175em}\to\hspace{-0.175em} \mathbb{R}^d\times \mathbb{R}^d$ \hspace{-0.1mm}is \hspace{-0.1mm}defined \hspace{-0.1mm}by \hspace{-0.1mm}$(X^{\theta_0}(T;\cdot),\widetilde{X}_T^{\theta_0})(x)\hspace{-0.15em}\coloneqq\hspace{-0.175em} (X^{\theta_0}(T;x),\widetilde{X}_T^{\theta_0}(x))$~\hspace{-0.1mm}for~\hspace{-0.1mm}all~\hspace{-0.1mm}${x\hspace{-0.175em}\in\hspace{-0.175em} \mathbb{R}^d}$.
Then, denoting by $\operatorname{pr}_i\colon \mathbb{R}^d\times \mathbb{R}^d\to \mathbb{R}^d$, $i\in \{1,2\}$, the projection onto the first $d$ or last $d$~\mbox{components}, respectively, \textit{i.e.}, $\operatorname{pr}_i(x,y)\hspace{-0.15em}\coloneqq\hspace{-0.15em} x$ if $i\hspace{-0.15em}=\hspace{-0.15em}1$ and $\operatorname{pr}_i(x,y)\hspace{-0.15em}\coloneqq \hspace{-0.15em}y$ if $i\hspace{-0.15em}=\hspace{-0.15em}2$ for all $(x,y)\in \mathbb{R}^d\hspace{-0.15em}\times\hspace{-0.15em} \mathbb{R}^d$,  
using Proposition~\ref{prop:characteristic_flow_representation}, equivalently \eqref{eq:deterministic_pushforward}, and the composition rule for pushforward measures (\textit{cf}.~\mbox{\cite[Sec.~5.2, Eq.~(5.2.4)]
{AmbrosioGigliSavare2008}}), shows that
$[\operatorname{pr}_1]_{\#}\pi_T^{\theta_0}=[\operatorname{pr}_1\circ (X^{\theta_0}(T;\cdot),\widetilde{X}_T^{\theta_0})]_{\#}\mu_0=[X^{\theta_0}(T;\cdot)]_{\#}\mu_0=\mu_T^{\theta_0}$ in $\mathcal{P}(\mathbb{R}^d)$ (\textit{i.e.},  the first marginal of $\pi_T^{\theta_0}$ is $\mu_T^{\theta_0}$),
while
\eqref{eq:first_order_target_measure} and the composition rule for pushforward measures~show~that $[\operatorname{pr}_2]_{\#}\pi_T^{\theta_0}=[\operatorname{pr}_2\circ (X^{\theta_0}(T;\cdot),\widetilde{X}_T^{\theta_0})]_{\#}\mu_0=[\widetilde{X}_T^{\theta_0}]_{\#}\mu_0=\widetilde{\mu}_T^{\theta_0}$ in $\mathcal{P}(\mathbb{R}^d)$ (\textit{i.e.},  its second marginal
is $\widetilde{\mu}_T^{\theta_0}$). Consequently, we have that $\pi_T^{\theta_0}\in\Pi(\mu_T^{\theta_0},\widetilde{\mu}_T^{\theta_0})$ and,  for every $p\in\{1,2\}$,  from the definition of the $p$-Wasserstein distance \eqref{eq:wasserstein_distance}, the uniform bound  
\eqref{eq:uniform_second_order_map_estimate}, and $\mu_0(\mathbb{R}^d)=\|\rho_0\|_{L^1(\mathbb{R}^d)}=1$, we infer that\vspace{-1mm}
\begin{align}\label{eq:wasserstein_coupling_estimate}
    \begin{aligned} 
	W_p(\mu_T^{\theta_0},\widetilde{\mu}_T^{\theta_0})
	&\leq
	\left(
		\int_{\mathbb{R}^d\times\mathbb{R}^d}
		|y-z|^p\,\mathrm{d}\pi_T^{\theta_0}(y,z)
	\right)^{1/p}
	\\
	&=
	\left(
		\int_{\mathbb{R}^d}
		|X^{\theta_0}(T;x)-\widetilde{X}_T^{\theta_0}(x)|^p\,\mathrm{d}\mu_0(x)
	\right)^{1/p} 
	\\
	&\leq
	\tfrac{T^2}{2}\|\mathbf{a}^{\theta_0}\|_{L^\infty(\mathbb{R}^d\times(0,T);\mathbb{R}^d)}\,,
     \end{aligned}\\[-6.5mm]\notag
\end{align}
which proves \eqref{eq:wasserstein_second_order_estimate}.
\end{proof}

\begin{remark}[Meaning and scope of Theorem \ref{thm:second_order_transport_realization}]
\label{rem:second_order_transport_accuracy}
Suppose that the solution exists on $[0,T_0]$ and that
$\|\mathbf{a}^{\theta_0}\|_{L^\infty(\mathbb{R}^d\times(0,T_0);\mathbb{R}^d)}<+\infty$. Applying
Theorem~\ref{thm:second_order_transport_realization} on every
$[0,T]\subseteq[0,T_0]$, for every $p\in\{1,2\}$, gives\vspace{-0.5mm}
\begin{alignat}{2}
	\|X^{\theta_0}(T;\cdot)-\widetilde{X}_T^{\theta_0}\|_{L^\infty(\mathbb{R}^d;\mathbb{R}^d)}
	+
	W_p(\mu_T^{\theta_0},\widetilde{\mu}_T^{\theta_0})
	&=
	\mathcal O(T^2)
	&&\quad(T\to0)\,.
	\label{eq:second_order_transport_asymptotics}\\[-6mm]\notag
\end{alignat}
Thus, the initial phase reproduces the prescribed potential transport
direction exactly to first order in time, while the density-dependent
material acceleration \eqref{eq:material_acceleration} contributes the
second-order remainder. The result is a local consistency statement
for the near-identity map \eqref{eq:first_order_transport_map}; it does
not assert exact or global reachability of an arbitrary target
measure.\enlargethispage{5mm}
\end{remark}

The Gaussian analysis and the local result above provide complementary
benchmarks: the former gives exact terminal reachability within an
invariant family, whereas the latter quantifies the short-time error
for general nonlinear potential transport directions. Next, we use the
explicit Gaussian solution to assess the numerical phase-identification
procedure.\newpage

\section{Numerical Validation}\enlargethispage{5mm}
\label{sec:num_validation}\vspace{-1mm}

\hspace{5mm}In this section, we investigate the numerical realization of the
PDE-constrained phase-identification framework introduced in
Subsection~\ref{subsec:pde_constrained_phase_identification}. First, 
we describe the discretization and optimization procedure common
to all experiments. Then, we present one- and two-dimensional numerical
tests, for which only the respective experimental setups and results
need to be specified.\vspace{-0.5mm}

\subsection{Numerical implementation and phase identification}
\label{subsec:numerical_implementation}\vspace{-0.5mm}

\hspace{5mm}Let $\Omega\coloneqq(-L,L)^d\subseteq\mathbb{R}^d$, with
$L>0$ and $d\in\mathbb N$, and, for $N_x\in \mathbb{N}$, let
$\mathcal{T}_h=\{K_{\boldsymbol{i}}\}_{\boldsymbol{i}\in\mathcal{I}_h}$~be~its~uniform Cartesian partition into $N_x^d$ cells of side length
$h\coloneqq \frac{2L}{N_x}$, indexed by 
$\boldsymbol{i}=(i_1,\ldots,i_d)\in
\mathcal{I}_h\coloneqq\{1,\ldots,N_x\}^d$.  The cell centers are denoted by
$x_{\boldsymbol{i}}\hspace{-0.1em}\coloneqq\hspace{-0.1em}
(-L+(i_\ell-\frac12)h)_{\ell=1}^d\hspace{-0.1em}\in\hspace{-0.1em} K_{\boldsymbol{i}} $, $\boldsymbol{i}\hspace{-0.1em}\in\hspace{-0.1em}
\mathcal{I}_h$,  and form the grid
$\mathcal G_h\hspace{-0.1em}\coloneqq\hspace{-0.1em}
\{x_{\boldsymbol{i}}\mid\boldsymbol{i}\hspace{-0.1em}\in\hspace{-0.1em}\mathcal{I}_h\}$.\linebreak
We identify scalar cell-centered grid functions with
$\mathbb{V}_h\coloneqq\mathbb{R}^{\smash{\mathcal{I}_h}}\cong\mathbb{R}^{\smash{N_x^d}}$
and equip this space with the discrete inner product
$(u_h,w_h)_h\hspace{-0.1em}\coloneqq\hspace{-0.1em}
h^d\sum_{\boldsymbol{i}\in\mathcal{I}_h}
u_{\boldsymbol{i}}w_{\boldsymbol{i}}$ and the induced norm
$\|u_h\|_h\hspace{-0.1em}\coloneqq\hspace{-0.1em}\smash{(u_h,u_h)_h^{1/2}}$,~${u_h,w_h\hspace{-0.1em}\in \hspace{-0.1em}\mathbb{V}_h}$.
Writing \hspace{-0.1mm}$\mathbf1_h\hspace{-0.15em}\in\hspace{-0.15em} \mathbb{V}_h$ \hspace{-0.1mm}for \hspace{-0.1mm}the \hspace{-0.1mm}constant \hspace{-0.1mm}grid \hspace{-0.1mm}function \hspace{-0.1mm}with \hspace{-0.1mm}value \hspace{-0.1mm}one, \hspace{-0.1mm}we
\hspace{-0.1mm}denote \hspace{-0.1mm}by
\hspace{-0.1mm}${\mathbb{V}_{h,0}\hspace{-0.15em}\coloneqq\hspace{-0.15em}
\{q_h\hspace{-0.15em}\in\hspace{-0.15em}\mathbb{V}_h\mid(q_h,\mathbf1_h)_h\hspace{-0.15em}=\hspace{-0.15em}0\}}$
the normalized phase space and by
$\mathbb{P}_h\hspace{-0.15em}\coloneqq\hspace{-0.15em}
\{r_h\hspace{-0.15em}\in\hspace{-0.15em}\mathbb{V}_h\mid
r_{\boldsymbol{i}}\hspace{-0.15em}\geq\hspace{-0.15em}0,\
(r_h,\mathbf1_h)_h\hspace{-0.15em}=\hspace{-0.15em}1\}$
the~discrete~probability~simplex.

The continuity equation \eqref{eq:model_density} is approximated by a 
conservative finite-volume scheme using Rusanov numerical fluxes
(\textit{cf}.~\cite[Chaps.~4~\&~12]{LeVeque2002}) and the quantum
Hamilton--Jacobi~equation~\eqref{eq:model_phase}~by~a~mono\-tone Godunov--Lax--Friedrichs finite-difference scheme constructed
from one-sided~\mbox{differences}~(\textit{cf}.~\mbox{\cite{KurganovNoellePetrova2001,OsherShu1991}}).
The resulting coupled semi-discrete system is advanced in time by the
third-order~\mbox{SSP--RK} method subject to a CFL restriction
(\textit{cf}.~\cite{GottliebShuTadmor2001}).
For the 
fully specified coupled discretization,~see~also~\mbox{\cite[Secs.~3~\&~5]{antil2026polarization}}.
All reported experiments use the exact Rusanov viscosity coefficient
and, after each SSP--RK3 stage, apply the density floor
$\rho_{\min}=10^{-20}\max_{\boldsymbol{i}\in \mathcal{I}_h}{\{(\rho_{0,h})_{\boldsymbol{i}}\}}$ and project the discrete phase function~to~zero~mean. No mass renormalization or clipping of the quantum potential is used; the
unit-mass condition in $\mathbb{P}_h$ is understood up to numerical tolerance
(observed relative terminal mass errors below $3\times10^{-13}$).

Moreover, let $q_h\hspace{-0.1em}\in\hspace{-0.1em}\mathbb{V}_{h,0}$ represent the scaled initial phase
$q\coloneqq\frac{1}{m}\theta_0\colon \hspace{-0.1em}\Omega\hspace{-0.1em}\to\hspace{-0.1em} \mathbb{R}$, so that
$\theta_{0,h}\coloneqq m q_h\hspace{-0.1em}\in\hspace{-0.1em} \mathbb{V}_{h,0}$. Then, the fully discrete
forward solver induces the terminal-time control-to-state map $\mathcal{S}_{T,h}\colon\mathbb{V}_{h,0}\to\mathbb{P}_h$, for every $q_h\in\mathbb{V}_{h,0}$ given via\vspace{-1mm}
\begin{align} 
	\smash{\mathcal{S}_{T,h}(q_h)
	\coloneqq
	\rho_h(\cdot,T;q_h)\quad \text{ in }\mathbb{P}_h\,.}
	\label{eq:discrete_terminal_control_to_state_map}
\end{align}
In particular, the grid values of $q_h\in\mathbb{V}_{h,0}\cong
\smash{\mathbb{R}^{\smash{N_x^d-1}}}$ form the finite-dimensional control parameters in
the discrete counterpart of
\eqref{eq:finite_dimensional_phase_identification} (\textit{i.e.}, we have that $\mathcal{A}=\mathbb{V}_{h,0}$ containing parameters $\alpha=q_h$).

As a discrepancy functional, we employ the discrete Hellinger
distance $\mathfrak D_h\colon \mathbb{P}_h\times \mathbb{P}_h\to \mathbb{R}_{\ge 0}$, defined for every  $r_h,s_h\in\mathbb{P}_h$ by\vspace{-1mm}
\begin{align}\label{eq:hellinger_distance}
    \smash{\mathfrak D_h(r_h,s_h)
\coloneqq\tfrac{1}{\sqrt{2}}\|\sqrt{r_h}-\sqrt{s_h}\|_h\,,}
\end{align}
where the square root is understood componentwise. Given a discrete target probability~density~${\rho_{\star,h}\in\mathbb{P}_h}$, we minimize the stabilized discrete cost functional $\mathcal{J}_h\colon \mathbb{V}_{h,0}\to \mathbb{R}_{\ge 0}$, defined for every $q_h\in\mathbb{V}_{h,0}$ by\vspace{-0.5mm}
\begin{align}
	\smash{\mathcal J_h(q_h)
	\coloneqq
	\mathfrak D_h^2(\mathcal{S}_{T,h}(q_h),\rho_{\star,h})
	+
	\lambda_{\mathrm s}\|\nabla_hq_h\|_h^2
	+
	\lambda_{\mathrm c}\|\mathrm D_h^2q_h\|_h^2\,,}
	\label{eq:discrete_reduced_cost_functional}\\[-6mm]\notag
\end{align}
where $\nabla_h$ and $\mathrm D_h^2$ denote the centered discrete
gradient and Hessian, respectively, and $\lambda_{\mathrm s},\lambda_{\mathrm c}\in \mathbb{R}_{>0}$.\vspace{-0.5mm}

\begin{algorithm}[H]
\caption{Discrete PDE-constrained phase-identification procedure}
\label{alg:discrete_phase_identification}
\begin{algorithmic}[1]
\Require $\rho_{0,h},\rho_{\star,h}\in \mathbb{P}_h$, $q_h^{(0)}\in \mathbb{V}_{h,0}$, $T,\lambda_{\mathrm s},\lambda_{\mathrm c}\in \mathbb{R}_{>0}$
\State Enforce $(q_h^{(0)},\mathbf 1_h)_h=0$
\For{$k=0,1,\ldots$}
	\State Set $\theta_{0,h}^{(k)}\gets m q_h^{(k)}$;
	solve the discrete free Madelung problem with initial data $(\rho_{0,h},\theta_{0,h}^{(k)})$~up~to~$T$
	\State Set
	$\rho_{T,h}^{(k)}
	\gets\mathcal{S}_{T,h}(q_h^{(k)})$;
	evaluate $\mathcal J_h(q_h^{(k)})$
	\State \textbf{if} the stopping criterion is satisfied,
	\textbf{then break}
	\State Perform one nonlinear optimization step to obtain
$q_h^{(k+1)}\in\mathbb{V}_h$; enforce
$(q_h^{(k+1)},\mathbf1_h)_h=0$
\EndFor
\State \Return $\theta_{0,h}^{\star}=m q_h^{(k)}$ and
$\rho_{T,h}^{(k)}$
\end{algorithmic}
\end{algorithm}\pagebreak

The controls in Figures~\ref{fig:density_evolution} and~\ref{fig:density_evolution_2D} are quasi-Newton approximations computed
using reverse-mode automatic differentiation through the fully discrete
solver; neither 
stationarity nor local~\mbox{optimality}~is~claimed.\vspace{-0.5mm} 

\subsection{1D experiments}
\label{sec:num_gaussian}\vspace{-0.5mm}\enlargethispage{3mm}

\hspace{5mm}In this subsection, we consider two 1D experiments within the numerical
framework of Subsection~\ref{subsec:numerical_implementation}. First,
we validate the discrete forward solver against the explicit Gaussian
dynamics derived in Section~\ref{sec:exact_gaussian}. Second, we apply
Algorithm~\ref{alg:discrete_phase_identification}, the discrete
counterpart of the PDE-constrained formulation in
Subsection~\ref{subsec:pde_constrained_phase_identification}, to an
asymmetric bimodal Gaussian-mixture target and assess the terminal-density~match.

For both experiments, we solve the 1D free Madelung
system on $\Omega=(-L,L)$, $L=10$, with~$m=1$~and terminal time 
$T=0.3$.
Unless otherwise stated, we employ $N_x=301$ spatial cells~and~${N_t=4500}$~\mbox{uniform} time steps. A zero numerical flux is imposed for
the discrete Born probability density~at~the~boundary~$\partial\Omega$. The reference probability density $\rho_0\colon \mathbb{R}\to \mathbb{R}_{>0}$ is the standard
Gaussian probability density, \textit{i.e.}, we set 
\begin{align}
	\rho_0
	\coloneqq \mathcal{N}(0,1)\,,
	\label{eq:one_dimensional_reference_density}
\end{align}
and its restriction to $\Omega$ is
normalized on the grid, so that $\rho_{0,h}\in\mathbb{P}_h$.\smallskip

$\bullet$ \emph{Gaussian forward-solver validation.}\quad As the target probability  density $\rho_{\mathrm{val}}\colon \mathbb{R}\to \mathbb{R}_{>0}$ we prescribe a Gaussian probability density, \textit{i.e.}, we set
\begin{align}
	\rho_{\mathrm{val}}
	=
	\mathcal{N}(
		\boldsymbol{\mu}_{\mathrm{val}},
		\sigma_{\mathrm{val}}^2
	)\,,\qquad 
	(\boldsymbol{\mu}_{\mathrm{val}},\sigma_{\mathrm{val}})
	=(0.8,1.2)\,.
	\label{eq:gaussian_validation_target}
\end{align}
According to \eqref{eq:gaussian_initial_phase},
\eqref{eq:gaussian_center}, and \eqref{eq:gaussian_width}, the
corresponding quadratic initial phase function $\theta_{0,\mathrm{val}}\colon \Omega\to \mathbb{R}$, for every $x\in \Omega$, is given via
\begin{align}
	\smash{\theta_{0,\mathrm{val}}(x)
	=
	\tfrac{m a_0}{2\sigma_0}
	(x-\boldsymbol{\mu}_0)^2
	+
	m\mathbf u_0
	(x-\boldsymbol{\mu}_0)
	+
	c_0\,,}
	\label{eq:gaussian_validation_initial_phase}
\end{align}
where $\boldsymbol{\mu}_0=0$, $\sigma_0=1$, $\mathbf u_0
	=\frac{1}{T}(\boldsymbol{\mu}_{\mathrm{val}}
		-\boldsymbol{\mu}_0)=
	\frac{8}{3}$, and $a_0
	=
	\frac{1}{T}
		(\smash{\sqrt{
			\sigma_{\mathrm{val}}^2
			-
			(\frac{T}{2m\sigma_0})^2
		}}
		-\sigma_0)
	\approx 0.6353$.~Among~the two values of $a_0$ satisfying the width
relation \eqref{eq:gaussian_width} at $t=T$, we choose the positive
one, for~which~the Gaussian width initially increases. By the
constant-phase invariance discussed in
Subsection~\ref{subsec:hamiltonian_structure}, the additive constant
$c_0\hspace{-0.15em}\in\hspace{-0.15em}\mathbb{R}$ may be chosen such that the grid representation of
$q_{\mathrm{val}}\hspace{-0.15em}\coloneqq\hspace{-0.15em}\frac{1}{m}\theta_{0,\mathrm{val}}$~\mbox{satisfies}~${q_{h,\mathrm{val}}\hspace{-0.15em}\in \hspace{-0.15em}\mathbb{V}_{h,0}}$.

Let
$\rho_{T,h}^{\mathrm{val}}
\coloneqq\mathcal{S}_{T,h}(q_{h,\mathrm{val}})\in \mathbb{P}_h$
be the computed discrete terminal-time Born probability density~and~let\linebreak
$\rho_{\mathrm{val},h}\hspace{-0.175em}\in\hspace{-0.175em}\mathbb{P}_h$ \hspace{-0.1mm}denote \hspace{-0.1mm}the \hspace{-0.1mm}grid-normalized
\hspace{-0.1mm}restriction \hspace{-0.1mm}of \hspace{-0.1mm}\eqref{eq:gaussian_validation_target}. \hspace{-0.1mm}For \hspace{-0.1mm}$N_x\hspace{-0.175em}=\hspace{-0.175em}301$, \hspace{-0.1mm}the \hspace{-0.1mm}terminal \hspace{-0.1mm}errors~\hspace{-0.1mm}are~\hspace{-0.1mm}of~\hspace{-0.1mm}the~\hspace{-0.1mm}\mbox{orders} 
\begin{align}
	\|\rho_{T,h}^{\mathrm{val}}-\rho_{\mathrm{val},h}\|_h
   \approx 9.3\hspace{-0.15em}\times\hspace{-0.15em}10^{-3}\,,
\quad
\|\rho_{T,h}^{\mathrm{val}}-\rho_{\mathrm{val},h}\|_{1,h}
   \approx 2.1\hspace{-0.15em}\times\hspace{-0.15em}10^{-2}\,,\quad \mathcal{D}_h(\rho_{T,h}^{\mathrm{val}},\rho_{\mathrm{val},h})\approx1.2\hspace{-0.15em}\times\hspace{-0.15em}10^{-2}\,,\label{eq:gaussian_validation_errors}
\end{align}
where
$\|r_h\|_{1,h}\coloneqq
h\sum_{\boldsymbol{i}\in\mathcal I_h}|r_{\boldsymbol{i}}|$
for $r_h\in\mathbb{V}_h$.
Figure~\ref{fig:exactsolution} compares the computed and exact \mbox{terminal-time}~pro\-bability densities and shows  decay of the errors in
\eqref{eq:gaussian_validation_errors} under  mesh-refinement,~using~$\smash{N_t=\lceil4500(\frac{N_x}{301})^2\rceil}$, so that $\Delta t=\mathcal{O}(h^2)$ and temporal
errors do not dominate.\vspace{-3.5mm}

\begin{figure}[H]
	\centering
    \includegraphics[width=\textwidth]{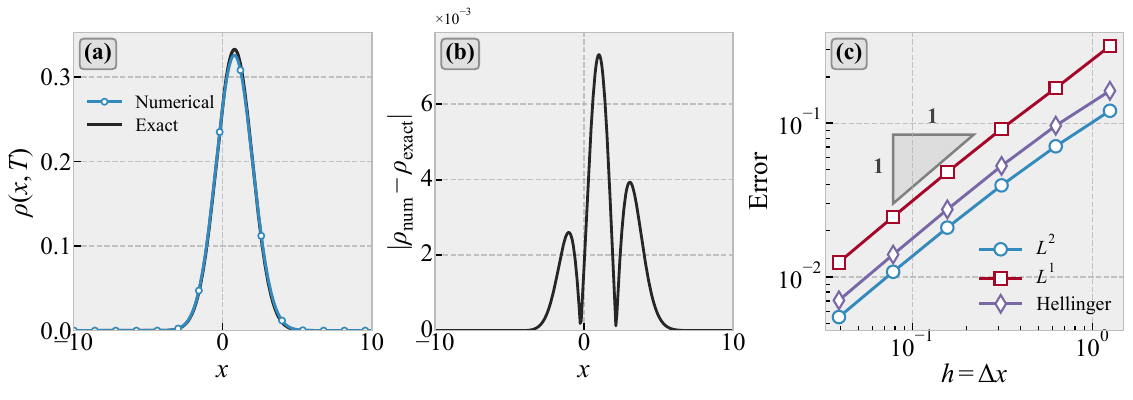}\vspace{-3.5mm}
    
	\caption{Validation of the discrete forward solver against the
	exact Gaussian solution~at~${T=0.3}$. \textbf{(a)} computed and exact
	terminal densities for $N_x=301$. \textbf{(b)} corresponding pointwise~\mbox{absolute}~error. \textbf{(c)} discrete $L^2$-, $L^1$-, and Hellinger errors
	under the uniform mesh refinements $N_x \in \{ 2^{i}\mid i=4,\ldots,9\} $, with the time step refined accordingly.}
	\label{fig:exactsolution}
\end{figure}

$\bullet$ \emph{Bimodal target and quantile-based initialization.}\quad
As the target probability density $\rho_\star\colon \mathbb{R}\to \mathbb{R}_{>0}$ we prescribe an asymmetric
bimodal Gaussian-mixture target, \textit{i.e.}, we set
\begin{align}
	\rho_\star
	= 
    (1-\omega)
	\mathcal{N}(
		\boldsymbol{\mu}_1,
		\sigma_1^2
	)+
	\omega
	\mathcal{N}(
		\boldsymbol{\mu}_2,
		\sigma_2^2
	)\,,
	\label{eq:one_dimensional_bimodal_target}
\end{align}
where $ \omega\hspace{-0.1em} = \hspace{-0.1em}\smash{\frac{2}{5}}$, $(\boldsymbol{\mu}_1,\sigma_1)
	\hspace{-0.1em}=\hspace{-0.1em}
	(4,0.5)$, and $(\boldsymbol{\mu}_2,\sigma_2)
	\hspace{-0.1em}=\hspace{-0.1em}
	(-3.6,1.5)$.
Note that, since the quadratic~phase~function
\eqref{eq:gaussian_initial_phase} preserves the isotropic Gaussian family, it
cannot generate the bimodal probability density~\eqref{eq:one_dimensional_bimodal_target}. Therefore, we  optimize the
initial phase function over the full discrete phase function space
$\mathbb{V}_{h,0}$.

In order to initialize Algorithm~\ref{alg:discrete_phase_identification}, we
use the monotone transport between the normalized restrictions of the
reference and target densities to $\Omega$. Their cumulative
distribution functions $F_0^L, F_\star^L\colon \Omega\to (0,1)$, for every $x\in \Omega$, are given via
\begin{align}
	\smash{F_0^L(x)
	\coloneqq
	\tfrac{1}{\|\rho_0\|_{L^1(\Omega)}}\textstyle\int_{-L}^{x}\rho_0(z)\,\mathrm dz\,,
	\qquad F_\star^L(x)
	\coloneqq
	\tfrac{1}{\|\rho_\star\|_{L^1(\Omega)}}\textstyle\int_{-L}^{x}\rho_\star(z)\,\mathrm dz\,.}
	\label{eq:truncated_distribution_functions}
\end{align}
The corresponding quantile transport map is
$\mathcal{T}_L\coloneqq\smash{(F_\star^L)^{-1}\circ F_0^L}\colon\Omega\to\Omega$
(\textit{cf}.~\cite[Sec.~2.1]{Santambrogio2015}); it maps equal
cumulative masses and pushes the normalized restriction of
$\mu_0=\rho_0\,\mathrm dx$ forward~to~that~of~${\mu_\star=\rho_\star\,\mathrm dx}$.

Motivated by the local transport estimates in
Theorem~\ref{thm:second_order_transport_realization} (see also
Remark~\ref{rem:second_order_transport_accuracy}), we choose the initial
velocity $\mathbf{v}_{0,\mathrm{init}}\in C^1(\Omega;\mathbb{R}^1)$
such that its first-order comparison map
$\widetilde{X}_T^{\theta_{0,\mathrm{init}}}\colon\Omega\to\Omega$
equals~$\mathcal{T}_L$,~namely,
\begin{align}
	\smash{\mathbf{v}_{0,\mathrm{init}}
	\coloneqq
	\tfrac{1}{T}(\mathcal{T}_L-\operatorname{id}_{\Omega})\quad\text{ in }\Omega\quad\Rightarrow\quad \widetilde X_T^{\theta_{0,\mathrm{init}}}
	\coloneqq
	\operatorname{id}_{\Omega}
	+
	T\mathbf v_{0,\mathrm{init}}
	=
	\mathcal{T}_L\quad\text{ in }\Omega\,.}
	\label{eq:quantile_initial_velocity}
\end{align}
Then, the associated scaled initial phase function
$q_{\mathrm{init}}\in C^2(\Omega)$ and initial phase function
$\theta_{0,\mathrm{init}}\in C^2(\Omega)$, respectively, are given via
\begin{align}
	\smash{q_{\mathrm{init}}(x)
	\coloneqq
	\textstyle\int_{-L}^{x}\mathbf{v}_{0,\mathrm{init}}(z)\,\mathrm dz
	+C_{\mathrm{init}}
	\quad\text{ for all }x\in\Omega\,,\qquad
	\theta_{0,\mathrm{init}}
	\coloneqq
	mq_{\mathrm{init}}\quad\text{ in }\Omega\,,}
	\label{eq:quantile_phase_initialization}
\end{align}
where $C_{\mathrm{init}}\in\mathbb{R}$ is chosen such that the
cell-centered grid representation of $q_{\mathrm{init}}$
satisfies $\smash{q_h^{(0)}}\in\mathbb{V}_{h,0}$. 

For comparison, the monotone transport map $\mathcal{T}\colon \mathbb{R}\to \mathbb{R}$ between the full-space
Gaussian probability measures
$\mathcal N(\boldsymbol{\mu}_0,\sigma_0^2)$ and
$\mathcal N(\boldsymbol{\mu}_\star,\sigma_\star^2)$ is affine, \textit{i.e.}, we have that
\begin{align*}
	\smash{\mathcal{T}
	=
	\boldsymbol{\mu}_\star
	+
	\tfrac{\sigma_\star}{\sigma_0}
	(\operatorname{id}_{\mathbb{R}}-\boldsymbol{\mu}_0)}\quad\text{ in }\mathbb{R}\,,
\end{align*}
and coincides with the terminal sampling map $X_\star\colon \mathbb{R}\to \mathbb{R}$ in
Corollary~\ref{cor:closed_form_gaussian_sampling}. By contrast,
the quantile transport map $\mathcal{T}_L\colon \Omega\to \Omega$ associated with the Gaussian-mixture target
\eqref{eq:one_dimensional_bimodal_target} is nonlinear and, therefore, yields a
spatially varying initial velocity field \eqref{eq:quantile_initial_velocity} and a nonquadratic initial phase function \eqref{eq:quantile_phase_initialization}.
It serves only to initialize
Algorithm~\ref{alg:discrete_phase_identification}; the subsequent
velocity field and characteristic transport are determined by the
discrete free Madelung evolution.\smallskip

$\bullet$ \emph{Full-grid phase identification.}\quad
Let $\rho_{\star,h}\in\mathbb{P}_h$ be the grid-normalized restriction
of the target probability density \eqref{eq:one_dimensional_bimodal_target}. Then, starting from the
quantile-based control $\smash{q_h^{(0)}}\in\mathbb{V}_{h,0}$ defined above, we
apply Algorithm~\ref{alg:discrete_phase_identification} to minimize
the discrete cost functional
\eqref{eq:discrete_reduced_cost_functional} with
$\lambda_{\mathrm s}=3\times10^{-6}$ and
$\lambda_{\mathrm c}=3\times10^{-7}$. Therefore, the control varies over
the full grid space $\mathbb{V}_{h,0}$ rather than a quadratic phase function
family.

Let $q_h^\star\in\mathbb{V}_{h,0}$ denote the computed scaled phase function,
$\theta_{0,h}^\star\coloneqq m q_h^\star\in \mathbb{V}_{h,0}$ the corresponding initial
phase function, and
$\rho_{T,h}^\star\coloneqq\mathcal{S}_{T,h}(q_h^\star)\in\mathbb{P}_h$ the resulting
terminal-time Born probability density. Figure~\ref{fig:density_evolution} shows the induced
density evolution.\vspace{-2.5mm}\enlargethispage{1.5mm}

\begin{figure}[H]
	\centering
    \includegraphics[width=\textwidth]{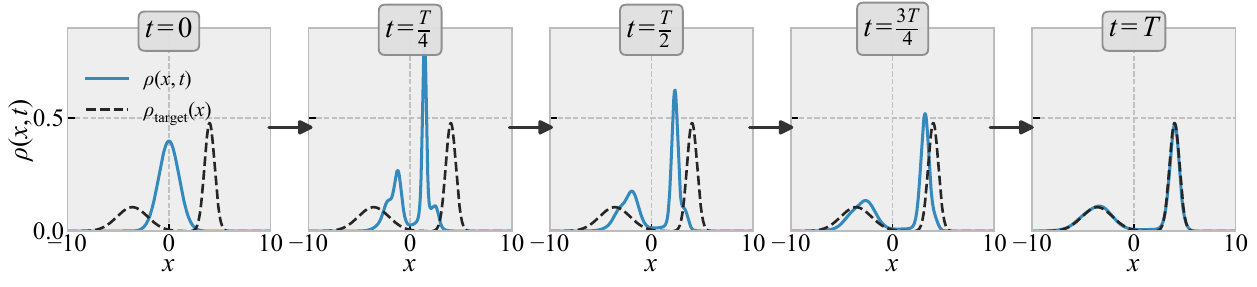}\vspace{-2.5mm} 
	\caption{Density evolution induced by the computed full-grid initial phase
 function $\theta_{0,h}^\star=mq_h^\star\in \mathbb{V}_{h,0}$, identified by minimizing the
PDE-constrained discrete cost functional
\eqref{eq:discrete_reduced_cost_functional}~to~\mbox{approximate}~the~\mbox{target} density $\rho_{\star,h}\hspace{-0.1em}\in\hspace{-0.1em} \mathbb{P}_h$ at $T\hspace{-0.1em}=\hspace{-0.1em}0.3$. The snapshots show
$t\hspace{-0.1em}=\hspace{-0.1em}0,\frac{T}{4},\frac{T}{2},\frac{3T}{4},T$ on a uniform grid~with~${N_x\hspace{-0.1em}=\hspace{-0.1em}301}$~cells.}
	\label{fig:density_evolution}
\end{figure}

The computed initial phase function $\theta_{0,h}^\star\in \mathbb{V}_{h,0}$ deforms the discrete unimodal reference density $\rho_{0,h}\in \mathbb{P}_h$ (\textit{cf}.\ \eqref{eq:one_dimensional_reference_density})
into an approximation of the asymmetric bimodal target $\rho_{\star,h}\in \mathbb{P}_h$ (\textit{cf}.\ \eqref{eq:one_dimensional_bimodal_target}), while the
subsequent velocity field is generated entirely by the discrete free
Madelung evolution.

The Gaussian benchmark in Figure~\ref{fig:exactsolution} indicates
linear convergence (in $\|\cdot\|_h$, $\|\cdot\|_{1,h}$,~and~the~discrete Hellinger distance \eqref{eq:hellinger_distance}) toward the exact dynamics, whereas
Figure~\ref{fig:density_evolution} shows~full-grid~phase~\mbox{identification}\linebreak for the non-Gaussian bimodal target probability density \eqref{eq:one_dimensional_bimodal_target}.~The~quantile~transport~map~${\mathcal{T}_L\colon\Omega\to\Omega}$
provides only the initial transport direction; the PDE-constrained
optimization adjusts the discrete initial phase function $\theta_{0,h}\in \mathbb{V}_{h,0}$ to the complete Madelung
evolution.~Thus,~only~the discrete initial phase function $\theta_{0,h}\in \mathbb{V}_{h,0}$ is optimized, while the
subsequent velocity field is generated~by~the~\mbox{governing}~\mbox{system}.\vspace{-0.5mm}

\subsection{2D experiment}
\label{subsec:two_dimensional_experiment}\vspace{-0.5mm}\enlargethispage{3.5mm}

\hspace{5mm}In this subsection, we consider a 2D experiment
within the numerical framework of~\mbox{Subsection}~\ref{subsec:numerical_implementation}. The objective is to
transport a centered isotropic Gaussian reference probability density into an
asymmetric bimodal Gaussian-mixture target probability density.
For this experiment, we solve the 2D free Madelung
\hspace{-0.1mm}system \hspace{-0.1mm}on \hspace{-0.1mm}$\Omega\hspace{-0.1em}=\hspace{-0.1em}(-L,L)^2$, $L\hspace{-0.1em}=\hspace{-0.1em}16$, \hspace{-0.1mm}with \hspace{-0.1mm}$m\hspace{-0.1em}=\hspace{-0.1em}1$ \hspace{-0.1mm}and \hspace{-0.1mm}terminal \hspace{-0.1mm}time~\hspace{-0.1mm}${T\hspace{-0.1em}=\hspace{-0.1em}0.3}$.~\hspace{-0.1mm}\mbox{Unless}~\hspace{-0.1mm}\mbox{otherwise}~\hspace{-0.1mm}stated, we employ $N_x=91$ spatial cells per coordinate
direction and
$N_t=900$ uniform time steps.  The reference probability density $\rho_0\colon \mathbb{R}^2\to \mathbb{R}_{>0}$ is the isotropic
Gaussian probability density
\begin{align}
	\rho_0
	\coloneqq
	\mathcal N(\boldsymbol{\mu}_0,\sigma_0^2I_2)\,,	\label{eq:two_dimensional_reference_density}
\end{align}
where $(\boldsymbol{\mu}_0,\sigma_0)
	=
	(\smash{(0,0)^\top},2)$, and its restriction to $\Omega$ is normalized on the grid, so that $\rho_{0,h}\in\mathbb{P}_h$.
As the target probability density, we prescribe the asymmetric
bimodal Gaussian mixture 
\begin{align}
	\rho_\star
	\coloneqq
	(1-\omega)
	\mathcal N(\boldsymbol{\mu}_1,\sigma_1^2I_2)
	+
	\omega
	\mathcal N(\boldsymbol{\mu}_2,\sigma_2^2I_2)\,,
	\label{eq:two_dimensional_target_density} 
\end{align}
where $\omega=\frac{2}{3}$, $(\boldsymbol{\mu}_1,\sigma_1)
	=(\smash{(4,4)^\top},1)$, and $(\boldsymbol{\mu}_2,\sigma_2)
	=(\smash{(-8,-8)^\top},1)$, and its restriction to $\Omega$ is normalized on the grid, so that $\rho_{\star,h}\in\mathbb{P}_h$. 

    $\bullet$ \emph{Rotated quantile-based initialization.}\quad
In order to initialize Algorithm~\ref{alg:discrete_phase_identification}, we exploit that the means $\boldsymbol{\mu}_1$ and $\boldsymbol{\mu}_2$ of the
two Gaussian components in \eqref{eq:two_dimensional_target_density} lie on the diagonal $x_1=x_2$ and
introduce the rotated coordinates 
\begin{align*}
	z=(z_1,z_2)
	\coloneqq
	\smash{\tfrac{1}{\sqrt2}}(x_1+x_2,x_1-x_2)\,.
\end{align*}
Along the $z_1$-direction, the \(z_1\)-marginals of the probability densities 
\eqref{eq:two_dimensional_reference_density} and \eqref{eq:two_dimensional_target_density} have the densities
$\varrho_{0,1}=\mathcal N(0,2^2)$ and
$\varrho_{\star,1}=(1-\omega)\mathcal N(4\sqrt2,1^2)
+\omega\mathcal N(-8\sqrt2,1^2)$, respectively. Denoting their
cumulative distribution functions by
$F_{0,1},F_{\star,1}\colon\hspace{-0.15em}\mathbb{R}\hspace{-0.15em}\to\hspace{-0.15em}(0,1)$, we define, as in the
1D experiment,~the~monotone~quantile\linebreak transport map
$\mathcal{T}_1\coloneqq F_{\star,1}^{-1}\circ F_{0,1}
\colon\mathbb{R}\to \mathbb{R}$. Along the transverse
$z_2$-direction, the marginal probability densities are
$\mathcal N(0,2^2)$ and $\mathcal N(0,1^2)$, so their monotone affine
transport map is
$\mathcal{T}_2\coloneqq (z_2\mapsto \frac{z_2}{2})\colon\mathbb{R}\to\mathbb{R}$.
Then, the first-order comparison map in these rotated coordinates is given
via $(T_1,T_2):\mathbb{R}^2\to\mathbb{R}^2$, and the associated initial
velocity field, expressed in rotated coordinates, is given via $\mathbf{v}_0\coloneqq (z\mapsto \frac{1}{T}(\mathcal{T}_1(z_1)-z_1,$ $\mathcal{T}_2(z_2)-z_2)^\top\colon \mathbb{R}^2\to \mathbb{R}^2$. 
Since the coordinate transformation is orthogonal~and~each~component~depends only on the corresponding coordinate, the resulting velocity field
is potential. 
A corresponding scaled initial phase function
$q_{\mathrm{init}}\colon\Omega\to\mathbb{R}$ and initial phase function
$\theta_{0,\mathrm{init}}\colon\Omega\to\mathbb{R}$, respectively,~for~every~$x\in \Omega$, are given by
\begin{align}
	q_{\mathrm{init}}(x_1,x_2)
	\coloneqq
	\tfrac1T
	\bigl(
		\textstyle\int_0^{z_1}
		(\mathcal{T}_1(s)-s)\,\mathrm ds
		-\tfrac14z_2^2
	\bigr)
	+C_{\mathrm{init}}\,,
	\qquad
	\theta_{0,\mathrm{init}}
	\coloneqq
	mq_{\mathrm{init}}\quad\text{ in }\Omega\,,
	\label{eq:two_dimensional_phase_initialization} 
\end{align}
where $C_{\mathrm{init}}\hspace{-0.15em}\in\hspace{-0.15em}\mathbb{R}$ is chosen such that the
cell-centered representation satisfies $\smash{q_h^{(0)}}\hspace{-0.15em}\in\hspace{-0.15em} \mathbb{V}_{h,0}$.~As~in~the~1D~experi\-ment, this construction only
initializes Algorithm~\ref{alg:discrete_phase_identification}; it is
not imposed as the terminal transport map.\smallskip

$\bullet$ \emph{Full-grid phase identification.}\quad
Starting from the quantile-based control $q_h^{(0)}\in V_{h,0}$~defined~above, we apply Algorithm~\ref{alg:discrete_phase_identification} to minimize~\eqref{eq:discrete_reduced_cost_functional}
with $\lambda_{\mathrm{s}}=3\times10^{-8}$ and
$\lambda_{\mathrm{c}}=3\times10^{-9}$. As shown in
Figure~\ref{fig:density_evolution_2D}, the computed initial phase function splits
the unimodal reference probability density into two components with
unequal masses and transport distances.  As in the 1D experiment,
only the initial phase function is optimized, whereas the subsequent velocity
field and Born probability density evolution are generated by the discrete
free Madelung system.

\begin{figure}[H]
	\centering
	\includegraphics[width=\textwidth]{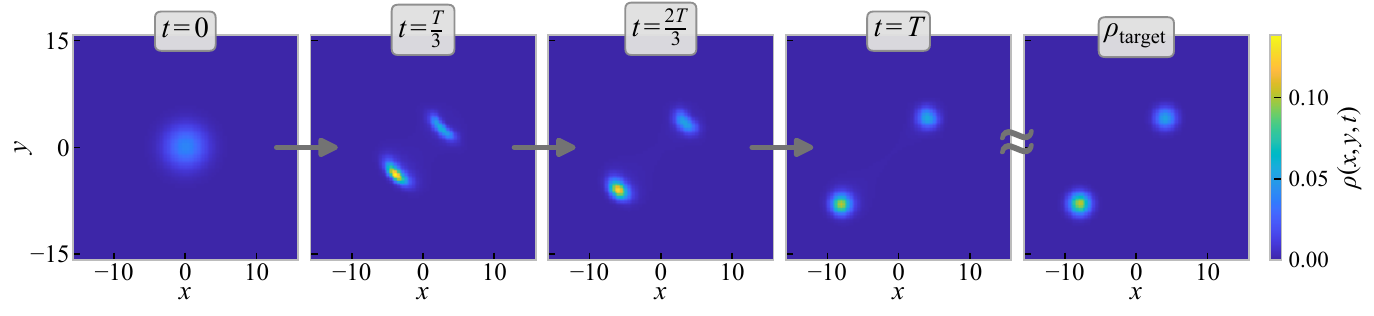}\vspace{-2.5mm}
	\caption{2D density evolution induced by the computed full-grid 
initial phase $\theta_{0,h}^\star=mq_h^\star\in \mathbb{V}_{h,0}$, identified by minimizing
the PDE-constrained discrete cost functional
\eqref{eq:discrete_reduced_cost_functional} to approximate the target probability density ${\rho_{\star,h}\in \mathbb{P}_h}$ at $T=0.3$. On a $91\times91$ grid, the first
four panels show $t=0,\frac{T}{3},\frac{2T}{3},T$, while the last panel
shows the target density. All panels use the same color scale.}
	\label{fig:density_evolution_2D}
\end{figure}

$\bullet$ \emph{Code availability.}
The MATLAB implementation and compact reference data used to reproduce
Figures~\ref{fig:exactsolution}--\ref{fig:density_evolution_2D} are archived as version~v1.0.0 on Zenodo
\cite{antil_2026_22101077}. The corresponding development repository is
available at
\url{https://github.com/alexkaltenbach/qh-gem-matlab}.

{\setlength{\bibsep}{0pt plus 0.0ex}\small
		
		\bibliographystyle{aomplain}
		\bibliography{references}
		
	}

\end{document}